\documentclass{amsart}
\usepackage{amsthm, amssymb, mathrsfs}
\usepackage{hyperref}
\hypersetup{
    pdftitle={Generalized Calabi-Yau manifolds and the chiral de Rham
    complex},    
    pdfauthor={Reimundo Heluani, Maxim Zabzine},     
    pdfkeywords={Generalized Calabi-Yau, chiral}, 
    colorlinks=true,       
    linkcolor=blue,          
    citecolor=blue,        
    filecolor=blue,      
    urlcolor=blue           
}

\DeclareMathOperator{\End}{End}
\DeclareMathOperator{\lie}{Lie}
\DeclareMathOperator{\Id}{Id}
\DeclareMathOperator{\dive}{div}
\DeclareMathOperator{\tr}{tr}
\DeclareMathOperator{\ad}{ad}
\def\vac{|0\rangle}
\def\cH{{\mathscr H}}
\def\cR{{\mathscr R}}

\def\cO{{\mathscr O}}

\def\fg{{\mathfrak g}}
\def\fh{{\mathfrak h}}

\newtheorem{thm}{Theorem}[section]
\numberwithin{equation}{section}
\newtheorem{prop}[thm]{Proposition}
\newtheorem{lem}[thm]{Lemma}

\theoremstyle{definition}
\newtheorem{defn}[thm]{Definition}
\newtheorem{ex}[thm]{Example}
\theoremstyle{remark}
\newtheorem{rem}[thm]{Remark}

\begin{document}
\title{Generalized Calabi-Yau manifolds and the chiral de Rham complex}
\author{Reimundo Heluani}
\address{813 Evans Hall dept. of Mathematics \\ University of California \\
Berkeley 94720}
\email{heluani@math.berkeley.edu}
\author{Maxim Zabzine}
\address{Department of Physics and Astronomy, Uppsala University\\ Box 803 SE-75108 Uppsala Sweden}
\email{Maxim.Zabzine@fysast.uu.se}
\begin{abstract}
	We show that the chiral de Rham complex of a generalized
	Calabi-Yau manifold carries $N=2$ supersymmetry. We discuss the
	corresponding topological twist for this $N=2$ algebra. We interpret
	this  
	 as an algebroid version of the super-Sugawara or
	Kac-Todorov construction. 
\end{abstract}
\maketitle

\section{Introduction}
In \cite{malikov} the authors introduced a sheaf $\Omega^{\mathrm{ch}}_M$ of super vertex algebras on any
smooth manifold $M$, they called it the \emph{chiral de Rham complex} of $M$. It
was subsequently studied both in the mathematics literature (cf. \cite{gerbes2},
\cite{heluani2}, \cite{frenkelnekrasov} among others) and the physics
literature, in connection to the $\sigma$-model with target $M$ (cf.
\cite{kapustin}, \cite{Witten} among others). In the algebraic context, when the
manifold $M$ has a global holomorphic volume form, it was shown in \cite{malikov}
that the cohomology $H^*(M, \Omega^{\mathrm{ch}}_M)$ of this sheaf, a super vertex
algebra, carries the structure of an $N=2$ superconformal vertex algebra. This
result was further generalized in the differential setting in \cite{heluani2} and
more recently in \cite{heluani8}, where it was shown that, when $M$ is Calabi-Yau,
$\Omega^\mathrm{ch}_M$ carries an $N=2$ superconformal structure associated to the
complex structure, and another $N=2$ structure associated to its symplectic
structure. Moreover, these two structures combine into two commuting $N=2$
superconformal structures on $\Omega^\mathrm{ch}_M$. 

On the other hand, it has been known for some time now in the physics literature,
that for the $\sigma$-model to posses $N=2$ supersymetry, the target manifold ought
to have the structure of a \emph{generalized complex manifold} (cf.
\cite{bredthauer},
\cite{zabzine}, \cite{zabzine2} and references therein). The aim of this article is to
show that, for this supersymmetry to subsist at the quantum level, the
\emph{canonical bundle} of $M$ has to be \emph{holomorphically trivial}. Formally we will show that given
a differentiable manifold $M$, to each pair $(\mathcal{J}, \varphi)$ where
$\mathcal{J}$ is a \emph{generalized complex structure} on $M$ and $\varphi$ is a
global closed pure spinor (a closed section of the canonical line bundle
$U_\mathcal{J}$ corresponding to $J$),
we will associate an $N=2$ superconformal structure on $\Omega^\mathrm{ch}_M$ of
central charge $c = 3 \dim_\mathbb{R} M$. This structure generalizes in the
Calabi-Yau case, those structures described in \cite{heluani2} and \cite{heluani8}.
 
We can perform a topological twist by reassigning the
conformal weights of the basic fermions in this theory. For example, in
the Complex case, the authors of \cite{malikov} declared the conformal
weight of fields corresponding to holomorphic forms to be zero, while the
conformal weight of fields corresponding to holomorphic vector fields is
$1$. We show that the twisting in the generalized complex case is a
generalization of both $A-$ and $B-$models. Indeed, one has to consider linear combinations of
differential forms and vector fields as the basic fermions in the twisted
theory. We identify the BRST cohomology of the chiral de Rham complex of
$M$ with the Lie algebroid cohomology of the corresponding Dirac
structure, obtaining thus another interpretation for the
Gerstenhaber algebra structure in the Lie algebroid cohomology of
a generalized Calabi-Yau manifold.

It is well known that given a simple or commutative
Lie algebra $\fg$ with an invariant bilinear form $(,)$, one can
construct an embedding of the $N=1$ super-vertex algebra in the
corresponding super-affine vertex algebra $V^k(\fg_{\mathrm{super}})$
(cf. \cite{kactodorov},
\cite{kac:vertex}). Taking Zhu algebras, one recovers the construction of
the \emph{cubic Dirac operator} of \cite{kostant2} (cf.
\cite{kacdesole2}). Our construction could be viewed as a groupoid
generalization of this construction. Loosely speaking, Courant algebroids
could be viewed as families of Lie algebras with invariant bilinear forms
$(,)$. Given a courant algebroid $E$, there exists a sheaf of SUSY vertex
algebras $U^\mathrm{ch}(E)$ constructed in a similar way to the sheaf of
twisted differential operators corresponding to a Lie algebroid. Choosing
special local frames for $E$, the superfield of $U^\mathrm{ch}(E)$
that generates supersymmetry is given by the same expression as in the
super-Sugawara or Kac-Todorov construction of \cite{kactodorov}.

The organization of this article is as follows. In section \ref{sec:basics} we
recall the basics of vertex algebra theory and SUSY vertex algebra theory, we refer
the reader to \cite{kac:vertex} for the former and \cite{heluani3} for the latter.
In section \ref{sec:generalized} we collect some results about Lie and
Courant algebroids, we briefly recall the definition of the
\emph{modular class} of a Lie algebroid, as well as we recall the basics of generalized complex
geometry following \cite{gualtieri1}. In
section \ref{sec:sheaf} we recall the construction of a sheaf of vertex algebras
associated to any \emph{Courant algebroid} on $M$. The chiral de Rham complex of
a differentiable manifold $M$ corresponds to the case when this Courant algebroid is
the standard algebroid $T_M \oplus T^*_M$. In this section we follow
\cite{heluani8}, while we refer the reader to the original literature for a
complete treatment \cite{malikov}, \cite{gerbes2}, \cite{bressler1}. In section
\ref{sec:sections} we construct global sections of the chiral de Rham complex of
$M$ associated to any pair $(\mathcal{J}, \varphi)$ as above. We state the main
results in this section (see Theorems \ref{thm:n=2thm} and
\ref{thm:quasi-iso}). The technical proofs and
computations can be found in the Appendix.

{\bf Acknowledgements:} 
We thank Francesco Bonechi, Gil Cavalcanti, Marco Gualtieri, Henrique
Burstyn and Alan Weinstein for stimulating discussions. We are happy to thank
the Kavli institute for theoretical physics and
the program \emph{geometrical aspects of string theory} at Nordita, where part of this work was
carried out. 

R.H. is 
supported by the Miller institute for basic research in science. 
M.Z. is supported by VR-grant 621-2004-3177 and by
VR-grant 621-2008-4273.
\section{Preliminaries on Vertex Algebras} \label{sec:basics}

\subsection{Vertex superalgebras} \label{classical}

In this section, we review the definition of vertex superalgebras,
as presented in \cite{kac:vertex}.
    Given a vector space $V$, an \emph{$\End(V)$-valued field} is a formal
    distribution of the form
    \begin{equation*}
        A(z) = \sum_{n \in \mathbb{Z}} z^{-1-n} A_{(n)},\qquad A_{(n)} \in
        \End(V),
    \end{equation*}
    such that for every $v \in V$, we have $A_{(n)}v = 0$ for large enough $n$.

    \begin{defn}
    A vertex super-algebra consists of the data of a super vector space $V$,
    an even vector $\vac \in V$ (the vacuum vector),
     an even endomorphism $T $, and a parity preserving linear map $A \mapsto Y(A,z)$ from
     $V$ to $\End(V)$-valued fields (the state-field correspondence). This
     data should satisfy the following set of axioms:
    \begin{itemize}
    \item Vacuum axioms:
        \begin{equation*}
            Y(\vac, z) = \Id, \qquad 
            Y(A, z) \vac = A + O(z), \qquad
            T \vac = 0.
        \end{equation*}
    \item Translation invariance:
        \begin{equation*}
            \begin{aligned}
                {[}T, Y(A,z)] &= \partial_z Y(A,z).
        \end{aligned}
        \end{equation*}
    \item Locality:
        \begin{equation*}
            (z-w)^n [Y(A,z), Y(B,w)] = 0, \qquad n \gg 0.
        \end{equation*}
     \end{itemize}
(The notation $O(z)$ denotes a power series in $z$ without constant
term.)
    \label{defn:1.3}
\end{defn}
    Given a vertex super-algebra $V$ and a vector $A \in V$, we expand the fields
    \begin{equation*}
        Y(A,z) = \sum_{{j \in \mathbb{Z}}} z^{-1-j}
        A_{(j)},
    \end{equation*}
    and we call the endomorphisms $A_{(j)}$ the \emph{Fourier modes} of
    $Y(A,z)$. Define now the operations:
    \begin{equation*}
            {[}A_\lambda B] = \sum_{{j \geq 0}}
            \frac{\lambda^{j}}{j!} A_{(j)}B, \qquad
            AB = A \cdot B := A_{(-1)}B.
    \end{equation*}
    The first operation is called the $\lambda$-bracket and the second is
    called the \emph{normally ordered product}.
     The $\lambda$-bracket contains all of the information about the commutators between the Fourier coefficients of fields in $V$. 

\begin{rem}
	Corresponding to a given a super-vertex algebra $V$, there exists
	an associative algebra $Z(V)$ called the \emph{Zhu algebra of V}.
	Below we will give some examples of these algebras and refer the
	reader to the classic literature in the subject  for its
	definition (see for example  \cite{kacdesole2}).
\end{rem}
\subsection{Examples}

In this section we review the standard description of the $N=1,2$
superconformal vertex algebras as well as the current or affine vertex
algebras. We describe the Sugawara and Kac-Todorov construction. In section \ref{sec:2}, the same
algebras will be described in the SUSY vertex algebra formalism.

\begin{ex}{\bf The $N=1$ (Neveu-Schwarz)
    superconformal vertex algebra} \label{N1ex}

    The $N=1$ superconformal vertex algebra $NS_c$ (\cite{kac:vertex}) of central charge $c$ is generated by two fields: $L(z)$, an even field of conformal weight $2$, and $G(z)$, an odd primary field of conformal weight $\frac{3}{2}$, with the $\lambda$-brackets
\begin{equation}
{[L}_\lambda L] = (T + 2\lambda) L + \frac{c \lambda^3}{12},
    \label{eq:1}
\end{equation}
\[
{[L}_\lambda G] = (T+\frac{3}{2} \lambda) G, \qquad
{[G}_\lambda G]  = 2L + \frac{c \lambda^2}{3}. 
\]
$L(z)$ is called the Virasoro field. The Zhu algebra $Z(NS_c)$ is the free
associative superalgebra in one odd generator $\mathbb{C}[\bar{G}]$. 
\end{ex}

\begin{ex}{\bf The $N=2$ superconformal vertex algebra} \label{N2ex}

The $N=2$ superconformal vertex algebra of central charge $c$ is
generated by the Virasoro field $L(z)$ with $\lambda$-bracket
(\ref{eq:1}), an even primary field $J(z)$ of conformal weight $1$,
and two odd primary fields $G^{\pm}(z)$ of conformal weight
$\frac{3}{2}$, with the remaining $\lambda$-brackets \cite{kac:vertex}
\begin{xalignat*}{2}
    {[J}_\lambda G^\pm] &= \pm G^\pm, & [J_\lambda J] &=
    \frac{c}{3}\lambda, \\
    {[G^+}_\lambda G^-] &= L + \frac{1}{2} TJ + \lambda J +
    \frac{c}{6}\lambda^2,
    & {[G^\pm}_\lambda G^\pm] &= 0.
\end{xalignat*}
\end{ex}

\begin{ex}{\bf The Universal affine vertex algebra} \label{excurrent1}
	Let $\fg$ be a simple or commutative Lie algebra with
	non-degenerate invariant bilinear form $(,)$. The universal
	affine vertex algebra $V^k(\fg)$, $k \in \mathbb{C}$ is generated
	by fields $a, b \in \fg$ with the following $\lambda$-bracket:
	\begin{equation*}
		[a_\lambda b] = [a,b] + k \lambda (a,b). 
	\end{equation*}
	Its corresponding Zhu algebra $Z(V^k(\fg))=U(\fg)$, the universal
	enveloping algebra of $\fg$. If $k \neq -h^\vee$,
	choosing dual bases $\{a^i\}$, $\{a_i\}$ for $(,)$  we can write
	the field
	\[ L := \frac{1}{2 (k + h^\vee)} a^i a_i, \]
	where $h^\vee$ is the dual Coxeter number of $\fg$ and we sum
	over repeated indexes. A simple
	computation shows that $L$ satisfies (\ref{eq:1}) and it is a
	\emph{superconformal vector}. Taking Zhu algebras for this
	morphism we find an embedding $C[x] \hookrightarrow U(\fg)$
	of a polynomial algebra in $U(\fg)$ mapping the generator $x$ to
	the Casimir element of $\fg$.
\end{ex}

\begin{ex}{\bf The super-affine vertex algebra} Let $\fg$ be as above,
	We have a super-vertex algebra generated by even fields $a,b \in
	\fg$ and corresponding odd fields $\bar{a}, \bar{b}$ with the
	following $\lambda$-brackets ($k \in \mathbb{C}$):
	\[ \begin{aligned}
		{[a}_\lambda b] &= [a,b] + \lambda (k + h^\vee) (a,b), &
		[\bar{a}_\lambda \bar{b}] &= (k + h^\vee) (a,b), \\
		[a_\lambda \bar{b}] &= [\bar{a}_\lambda b] =
		\overline{[a,b]}.
	\end{aligned} \]
	Let $\{a^i\}$, $\{a_i\}$ be dual bases as above. If $k \neq
	-h^\vee$, introduce the
	following odd field (cf. \cite{kacdesole2}):
	\[  G = \frac{1}{k + h^\vee} \left( a^i \overline{a_i} +
	\frac{1}{3(k+h^\vee)} \overline{[a^i, a^j]} (\overline{a_i}
	 \, \overline{a_j}) \right). \]
	\label{ex:2} Then $G$ generates the super-vertex algebra of
	Example \ref{N1ex}. Taking Zhu algebras we obtain the
	construction of the cubic Dirac operator of \cite{kostant2} (see
	\cite{kacdesole2}).
	\label{ex:super-currents1}
\end{ex}

\subsection{SUSY vertex algebras}\label{sec:2} In this section we collect some
results on SUSY vertex algebras from \cite{heluani3}. 

    Introduce formal variables $Z=(z,\theta)$ and $W =
    (w,\zeta)$, where $\theta, \zeta$ are odd
    anti-commuting variables and $z, w$ are even commuting variables.
    Given an integer $j$ and $J = 0$ or $1$ we put $Z^{j|J} = z^j \theta^J$.

    Let $\cH$ be the superalgebra generated by $\chi, \lambda$ with the relations
    $[\chi, \chi] = - 2 \lambda$, where $\chi$ is
    odd and $\lambda$ is even and central. We will consider another set of
    generators $-S, -T$ for $\cH$ where $S$ is odd, $T$ is central, and $[S, S]
    = 2 T$. Denote $\Lambda = (\lambda, \chi)$,
    $\nabla = (T, S)$, $\Lambda^{j|J} = \lambda^j \chi^J$ and $\nabla^{j|J} =
    T^j S^J$.

    Given a super vector space $V$ and a vector $a \in V$, we will denote by
    $(-1)^a$ its parity.
    Let $U$ be a vector space, a $U$-valued formal distribution is an
    expression of the form
    \begin{equation*}
        \sum_{\stackrel{j \in \mathbb{Z}}{J = 0,1}} Z^{-1-j|1-J} w_{(j|J)}
        \qquad w_{(j|J)} \in U.
    \end{equation*}
    The space of such distributions will be denoted by $U[ [Z, Z^{-1}] ]$. If
    $U$ is a Lie algebra we will say that two such distributions $a(Z), \,
    b(W)$ are
    \emph{local} if
    \begin{equation*}
        (z - w)^n [a(Z), b(W)] = 0 \qquad n \gg 0.
    \end{equation*}
    The space of distributions such that only finitely many negative powers
    of $z$ appear (i.e. $w_{(j|J)} = 0$ for large enough $j$) will be denoted
    $U( (Z ))$. In the case when $U = \End(V)$ for another vector space $V$,
    we will say that a distribution $a(Z)$ is a \emph{field} if $a(Z)v \in V(
    (Z ))$ for all $v \in V$.
    \begin{defn}[\cite{heluani3}]
    An $N_K=1$ SUSY vertex algebra consists of the data of a vector space $V$,
    an even vector $\vac \in V$ (the vacuum vector), an odd endomorphism
    $S$ (whose square is an even endomorphism we denote $T$),
    and a parity preserving linear map $A \mapsto Y(A,Z)$ from
     $V$ to $\End(V)$-valued fields (the state-field correspondence). This
     data should satisfy the following set of axioms:
    \begin{itemize}
    \item Vacuum axioms:
        \begin{equation*}
            Y(\vac, Z) = \Id, \qquad
            Y(A, Z) \vac = A + O(Z), \qquad
            S \vac = 0.
        \end{equation*}
    \item Translation invariance:
        \begin{equation*}
            \begin{aligned}
            {[} S, Y(A,Z)] &= (\partial_\theta - \theta \partial_z)
            Y(A,Z),\\
            {[}T, Y(A,Z)] &= \partial_z Y(A,Z).
        \end{aligned}
        \end{equation*}
    \item Locality:
        \begin{equation*}
            (z-w)^n [Y(A,Z), Y(B,W)] = 0, \qquad n \gg 0.
        \end{equation*}
     \end{itemize}
    \label{defn:2.3}
\end{defn}
\begin{rem}
    Given the vacuum axiom for a SUSY vertex algebra, we will use the state
    field correspondence to identify a vector $A \in V$ with its corresponding
    superfield $Y(A,Z)$.
    \label{rem:nosenosenose}
\end{rem}
    Given a $N_K=1$ SUSY vertex algebra $V$ and a vector $A \in V$, we expand the fields
    \begin{equation}
        Y(A,Z) = \sum_{\stackrel{j \in \mathbb{Z}}{J = 0,1}} Z^{-1-j|1-J}
        A_{(j|J)}, 
	\label{eq:fourier}
    \end{equation}
    and we call the endomorphisms $A_{(j|J)}$ the \emph{Fourier modes} of
    $Y(A,Z)$. Define now the operations:
    \begin{equation}
            {[}A_\Lambda B] = \sum_{\stackrel{j \geq 0}{J = 0,1}}
            \frac{\Lambda^{j|J}}{j!} A_{(j|J)}B, \qquad
            A B = A_{(-1|1)}B.
        \label{eq:2.4.2}
    \end{equation}
    The first operation is called the $\Lambda$-bracket and the second is
    called the \emph{normally ordered product}.
\begin{rem}
    As in the standard setting, given a SUSY VA $V$ and a vector $A \in V$, we
    have:
    \begin{equation*}
        Y(TA, Z) = \partial_z Y(A,Z) = [T, Y(A,Z)]. 
    \end{equation*}
    On the other hand, the action of the derivation $S$ is described
    by:
    \begin{equation*}
        Y(SA,Z) = \left( \partial_\theta + \theta \partial_z \right) Y(A,Z)
        \neq [S, Y(A,Z)].
    \end{equation*}
    \label{rem:caca5}
\end{rem}
The relation with the standard field formalism is as follows.
Suppose that $V$ is a vertex superalgebra  as defined in section
\ref{classical}, together with a homomorphism from the $N=1$
superconformal vertex algebra in example \ref{N1ex}. $V$ therefore
possesses an even vector $\nu$ of conformal weight $2$, and an odd
vector $\tau$ of conformal weight $\frac{3}{2}$, whose associated
fields
\begin{equation*}
    \begin{aligned}
 Y(\nu,z) &= L(z) = \sum_{n \in \mathbb{Z}} L_n z^{-n-2}, \\
Y(\tau,z) &= G(z) = \sum_{n \in 1/2 + \mathbb{Z}} G_n z^{-n -
\frac{3}{2}},
\end{aligned}
\end{equation*}
have the $\lambda$-brackets as in example \ref{N1ex}, and where we
require $G_{-1/2}=S$ and $L_{-1}=T$. We can then endow $V$ with the
structure of an $N_K=1$ SUSY vertex algebra via the state-field
correspondence \cite{kac:vertex}
\[
Y(A,Z) = Y^{c} (A,z) + \theta Y^{c}(G_{-1/2} A, z),
\]
where we have written $Y^{c}$ to emphasize that this is the
usual state-field (rather than state--superfield)
correspondence in the sense of section \ref{classical}.

There exist however $SUSY$ vertex algebras without such a map
from the $N=1$ superconformal vertex algebra.
\begin{defn}
    Let $\cH$ be as before. \emph{An $N_K=1$ SUSY Lie
    conformal  algebra} is a $\cH$-module $\cR$ with an operation
    $[\,_{\Lambda}\,]: \cR \otimes \cR \rightarrow \cH
    \otimes \cR$ of degree
    $1$ satisfying:
    \begin{enumerate}
        \item Sesquilinearity
            \begin{equation*}
                [S a_\Lambda b] =  \chi [a_\Lambda b],
                \qquad [a_\Lambda S b] = -(-1)^{a} \left(S
                + \chi
                \right) [a_\Lambda b].
            \end{equation*}
        \item Skew-Symmetry:
            \begin{equation*}
                [b_\Lambda a] =  (-1)^{a b} [b_{-\Lambda -
                \nabla} a].
            \end{equation*}
            Here the bracket on the right hand side is computed as
            follows: first compute $[b_{\Gamma}a]$, where $\Gamma =
            (\gamma, \eta)$ are generators of $\cH$ super commuting
            with $\Lambda$, then replace $\Gamma$ by $(-\lambda - T,
            -\chi - S)$.
        \item Jacobi identity:
            \begin{equation*}
                [a_\Lambda [b_\Gamma c]] = -(-1)^{a} \left[
                [ a_\Lambda b]_{\Gamma + \Lambda} c \right] +
                (-1)^{(a+1)(b+1)} [b_\Gamma [a_\Lambda c]],
            \end{equation*}
            where the first bracket on the right hand side is computed as in Skew-Symmetry
            and the identity is an identity in $\cH^{\otimes 2} \otimes\cR$.
    \end{enumerate}
    \label{defn:k.conformal.1}
    \end{defn}
    Given an $N_K=1$ SUSY VA, it is canonically an $N_K=1$ SUSY Lie conformal algebra with
    the bracket defined in (\ref{eq:2.4.2}). Moreover, given an $N_K=1$ Lie
    conformal algebra $\cR$, there exists a unique $N_K=1$ SUSY VA called the
    \emph{universal enveloping SUSY vertex algebra of $\cR$} with the property
    that if $W$ is another $N_K=1$ SUSY VA and $\varphi : \cR \rightarrow W$ is a
    morphism of Lie conformal algebras, then $\varphi$ extends uniquely to a
    morphism $\varphi: V \rightarrow W$ of SUSY VAs.
    The operations (\ref{eq:2.4.2}) satisfy:
    \begin{itemize}
        \item Quasi-commutativity:
            \begin{equation*}
                ab - (-1)^{ab} ba = \int_{-\nabla}^0 [a_\Lambda
                b] d\Lambda.
            \end{equation*}
        \item Quasi-associativity
            \begin{equation*}
                (ab)c - a(bc) = \sum_{j \geq 0}
                a_{(-j-2|1)}b_{(j|1)}c + (-1)^{ab} \sum_{j \geq 0}
                b_{(-j-2|1)} a_{(j|1)}c.
            \end{equation*}
        \item Quasi-Leibniz (non-commutative Wick formula)
            \begin{equation*}
                [a_\Lambda bc ] = [a_\Lambda b] c + (-1)^{(a+1)b}b
                [a_\Lambda c] + \int_0^\Lambda [ [a_\Lambda
                b]_\Gamma c] d \Gamma,
            \end{equation*}
    \end{itemize}
    where the integral $\int d\Lambda$ is $\partial_\chi \int d\lambda$. In
    addition, the vacuum vector is a unit for the normally ordered product
    and the endomorphisms $S, T$ are odd and even derivations respectively of
    both operations.
\subsection{Examples}
\begin{ex}
    Let $\cR$ be the free $\cH$-module generated by an odd vector $H$.
    Consider the following Lie conformal algebra structure in $\cR$:
    \begin{equation*}
        {[}H_\Lambda H] = (2T + \chi S + 3 \lambda) H.
    \end{equation*}
    This is the \emph{Neveu-Schwarz} algebra (of central charge 0). This
    algebra admits a central extension of the form:
    \begin{equation*}
        {[}H_\Lambda H] = (2T + \chi S + 3\lambda) H + \frac{c}{3} \chi
	\lambda^2,
    \end{equation*}
    where $c$ is any complex number. The associated universal enveloping
    SUSY VA is the \emph{Neveu-Schwarz} algebra of central charge
    $c$. If we
    decompose the corresponding field
    \begin{equation*}
        H(z,\theta) = G(z) + 2 \theta L(z),
    \end{equation*}
    then the fields $G(z)$ and $L(z)$ satisfy the commutation relations of
    the $N=1$ super vertex algebra of Example \ref{N1ex}.
    \label{ex:2.9}
\end{ex}
\begin{ex} \label{ex:2.11.a}
    The $N=2$ superconformal vertex algebra is generated by $4$
    fields \cite{kac:vertex}. In this context it is generated by two
    superfields -- an $N=1$ vector $H$ as in Example
    \ref{ex:2.9} and an even current $J$, primary of conformal weight $1$,
    that is:
    \begin{equation*}
        {[}H_\Lambda J] = (2T + 2\lambda + \chi S) J.
    \end{equation*}
    The remaining commutation relation is
    \begin{equation*}
        [J_\Lambda J] = - (H + \frac{c}{3} \lambda \chi).
    \end{equation*}
    Note that given
    the \emph{current} $J$ we can recover the $N=1$ vector $H$.
     In terms of the fields of Example \ref{N2ex}, $H, J$ decompose as

\begin{equation}
    \begin{aligned}
         J(z,\theta) &= - \sqrt{-1}J(z) - \sqrt{-1}\theta \left( G^-(z) - G^+(z)
        \right), \\
        H(z,\theta) &= \left( G^+(z) + G^-(z) \right) + 2 \theta L(z).
    \end{aligned}
    \label{eq:decomposeas}
\end{equation}

\end{ex}

\begin{ex}[{Super currents \cite[Thm. 5.9]{kac:vertex}, \cite[Ex.
	5.9]{heluani3}}] Let $\fg$ be a finite dimensional Lie algebra
	with non-degenerate invariant form $(,)$. We construct an $N_K=1$
	SUSY vertex algebra generated by odd superfields:
	\begin{equation}
		{[a}_\Lambda b] = [a, b] + \chi (k+h^\vee) (a, b), \qquad a,b \in
		 \Pi \fg, \quad k \in \mathbb{C}.
		\label{eq:super-currents}
	\end{equation}
	Recall \cite{kactodorov} (see also
\cite[Ex. 5.9]{heluani3}, \cite{kazama} in the superfield formalism) that
when $k \neq - h^\vee$, the
	superfield
	\begin{equation}
		H_0 = \frac{1}{k+h^\vee} \left( (Sa^i) a_i + \frac{1}{3
		(k+h^\vee)} a^i (a^j
		[a_i,a_j]) \right), 
	\label{eq:kt-g}
	\end{equation}
	where $\{a_i\}$ and $\{a^i\}$ are dual bases of $\fg$ with
	respect to $(,)$, generates an $N=1$ SUSY vertex algebra of central charge 
	\[ c_0 = \frac{k  \dim \fg }{k+h^\vee} + \frac{\dim \fg}{2}, \] as in
	Example \ref{ex:2.9}. Here $h^\vee$ is the dual Coxeter number of
	$\fg$. Moreover, for each $a \in \fg$, the corresponding
	superfield is primary of conformal weight $\frac{1}{2}$, namely:
	\[  {[H_0}_\Lambda a ] = (2T + \lambda + \chi S) a.\]
	 Given any $v \in \fg$ we can deform the field
	$H_0$ above as 
	\begin{equation}
		H = H_0 + T v, 
		\label{eq:hdeformed}
	 \end{equation}
	 and it is straightforward to
	show that this field generates the
	Neveu Schwarz algebra of central charge \[c = c_0 - 3 (k+h^\vee)
	(v, v).\]With respect to this superconformal vector, the fields
	$a \in \fg$ are no longer primary.
		\label{ex:super-currents}
\end{ex}
\subsection{Manin triples and $N=2$ structures} In this
	section we extend the $N=1$ structure of Example
	\ref{ex:super-currents} to
	an $N=2$ structure when $g = \fh \oplus
	\fh^*$ is the double of a Lie bialgebra. The construction here
	presented is a particular case of the construction of E. Getzler
	\cite{getzler1}. We include here the proofs since we will need an
	algebroid version of this below.
	Let $\{e_i\}$ be a basis for $\fh$
	and let $\{e^i\}$ be the dual basis for $\fh^*$. We let \[J :=
	\frac{i}{k+h^\vee} e^i e_i, \] be the even superfield of
	$V^k(\fg_{\mathrm{super}})$ corresponding to the standard
	R-matrix. Recall that $2h^\vee$ is the eigenvalue of the Casimir
	of $\fg$ in its adjoint representation. Note that the element $v'_\fh := [e^i, e_i] \in \fg $
	does not depend on the choice of basis. Decomposing $v'_\fh = w +
	w^*$, where $w \in \fh$ and $w^* \in \fh^*$, we define the
	element \[v_\fh = w - w^*,\] and with a simple computation we
	find 
	\begin{equation}
		(v_\fh, v_\fh) = - \frac{2}{3} h^\vee \dim \fh.
		\label{eq:normv}
	\end{equation}
	\begin{prop} \hfill 
		\begin{enumerate}
			\item $J$ satisfies $[J_\Lambda J] = - (H +
				\tfrac{c \lambda \chi}{3} )$ where $H$ is
				the odd superfield given by
	\begin{multline}
	H :=  \frac{1}{(k+h^\vee)^2}  \Bigl[ e^i \bigl( e^j[e_i,e_j] \bigr) +
	e_i \bigl( e_j [e^i, e^j] \bigr) \Bigr] 
	 +\\ \frac{1}{(k + h^\vee)} \left( e_j (Se^j)
	+ e^j Se_j \right) + \frac{1}{k+h^\vee} T v_\fh.
		\label{eq:hdefsind}
	\end{multline}
	and $c = 3 \dim \fh$.
\item The superfields $J$ and $H$ generate an $N=2$ vertex algebra of
	central charge $c=3 \dim \fh$  as in
	Example \ref{ex:2.11.a}.
		\end{enumerate}
		\label{prop:doubles}
	\end{prop}
	\begin{proof}
		Let $c_{ij}^k$ and $c^{ij}_k$ be the structure constants
		of $\fh$ and $\fh^*$ respectively in the bases $\{e_i\}$,
		$\{e^i\}$. 
		In order to compute $[J_\Lambda J]$ we start with:
		\begin{multline}
		{[e_j}_\Lambda e^ie_i] = \left({[e_j}, e^i] + 
		\chi (k+h^\vee) {\delta_j}^i  \right) e_i + e^i
		{[e_j},e_i]   +
		\int_0^\Lambda  \eta (k+h^\vee)  ([e_j,e^i],
		e_i) d\Gamma \\ =
		\left( c_j^{ik} e_k e_i - c^i_{jk} e^k e_i + c_{ji}^k e^i
		e_k \right)
		+  \chi (k + h^\vee) e_j  +
		\lambda (k + h^\vee) c^i_{ij} \\ = c^{ik}_j e_k e_i + \chi (k + h^\vee) e_j  +
		\lambda (k + h^\vee) c^i_{ij}.
\end{multline}
By skew-symmetry we obtain:
\begin{equation}
		{[e^ie_i}_\Lambda e_j] = c^{ik}_j e_k e_i - (\chi + S)
		(k + h^\vee) e_j - \lambda (k+h^\vee) c^i_{ij}.
\end{equation}
Similarly we have
\begin{multline}
		{[e^j}_\Lambda e^ie_i] = {[e^j}, e^i]
		 e_i +  e^i
		 \left( {[e^j},e_i] +  \chi (k+ h^\vee) {\delta^j}_i
		 \right) + 
		 \int_0^\Lambda  \eta (k + h^\vee) ([e^j,e^i], e_i) d\Gamma \\ =
		 \left( c^{ji}_k e^k e_i + c^j_{ik} e^ie^k - c^{jk}_i e^i
		 e_k\right)
		-  \chi (k+h^\vee) e^j + \lambda (k+h^\vee)
		c^{ji}_{i}, \\ = c^j_{ik} e^i e^k -  \chi (k+h^\vee) e^j + \lambda (k+h^\vee)
		c^{ji}_{i},
\end{multline}
and by skew-symmetry we obtain:
\begin{equation}
	{[e^i e_i}_\Lambda e^j] = c^j_{ik} e^i e^k + (\chi + S) (k +
	h^\vee) e^j - \lambda (k + h^\vee) c^{ji}_i.
	\end{equation}
Now we proceed using the non-commutative Wick formula:
\begin{multline}
	{[e^ie_i}_\Lambda e^j e_j] = \left( c^j_{ik} e^i e^k + (\chi + S)
	(k+h^\vee)e^j - \lambda (k+h^\vee) c^{ji}_i \right) e_j + \\ e^j \left(
	c^{ki}_j e_k e_i + (\chi +S) (k+ h^\vee) e_j + \lambda (k+h^\vee)
	c^i_{ij} \right) + \int_0^\Lambda c^j_{ik} [e^i e^k_\Gamma e_j]
	d\Gamma + \\ \int_0^\Lambda (\eta - \chi) (k+h^\vee) ([e^j, e_j] + \eta (k +
	h^\vee) \dim \fh ) d\Gamma.
	\label{eq:quseqwe}
\end{multline}
We can compute the integral term easily as:
\begin{equation*}
	 2 \lambda (k+h^\vee) c^i_{ji} e^j  + \lambda (k + h^\vee) [e^j,
	 e_j] + \lambda \chi (k+h^\vee)^2 \dim \fh,
	\end{equation*}
and replacing in (\ref{eq:quseqwe}) we obtain:
\begin{multline}
	{[e^ie_i}_\Lambda e^j e_j] =  \left( c^j_{ik} e^i e^k + (\chi + S)
	(k+h^\vee)e^j  \right) e_j + \\ e^j \left(
	c^{ki}_j e_k e_i + (\chi + S) (k+ h^\vee) e_j 
	\right) + \lambda \chi (k+h^\vee)^2 \dim \fh = \\ e^i (e^j [e_i,
	e_k] ) + 2 (k + h^\vee) c^{j}_{ij} T e^i + e_i (e_j [e^i, e^j])
	+\\
	(k + h^\vee) \left( e^j Se_j +  e_j Se^j \right) + (k + h^\vee) T
	[e^j, e_j] + \lambda \chi (k+h^\vee)^2 \dim \fh \\ = (k+ h^\vee)
	\left(e^j Se_j + e_j Se^j \right) + e^i (e^j
	[e_i, e_k]) + e_i (e_j [e^i, e^j]) + \\ (k+h^\vee) \left( c^j_{ij}
	Te^i + c^{ij}_j Te_i \right) + \lambda \chi (k + h^\vee)^2 \dim
	\fh,
\end{multline}
proving (1). 
In order to prove (2), a simple computation shows that $H$
can be written as  in (\ref{eq:hdeformed}) with $v =
\tfrac{v_\fh}{k+h^\vee}$. Therefore $H_0$ generates an $N=1$ algebra
of central charge $c_0 = \tfrac{2 k \dim \fh}{k+h^\vee} + \dim \fh$ and each
element of $\fg$ is a primary field of conformal weight $1/2$ with
respect to $H_0$
\cite{kacdesole2}.  It follows
from (\ref{eq:normv}) that the central charge of $H$ is given by $c = 3
\dim \fh$. We only need to prove that $J$ is primary of conformal weight
$1$. For this we compute:
\[ \begin{aligned} {[H}_\Lambda e^i e_i] &= \left( (2T + \lambda +
	\chi S)e^i - \frac{\lambda}{k + h^\vee} [v_\fh, e^i]  - \lambda \chi
	(v_\fh, e^i) \right) e_i  \\ & \quad+ e^i \left(
	(2T + \lambda + \chi S) e_i - \frac{\lambda}{k+h^\vee} [v_\fh,
	e_i] - \lambda \chi (v_\fh, e_i)
	\right) + \\ & \quad \int_0^\Lambda (-2 \gamma + \lambda -\chi
	\eta) ([e^i, e_i] + \eta (k + h^\vee) \dim \fh )  - \lambda \eta
	(v_\fh, [e^i, e_i]) d \Gamma, \\ &= (2T + 2\lambda + \chi S) e^ie_i
	- \frac{\lambda}{k+h^\vee} ([v_\fh, e^i] e_i + e^i [v_\fh, e_i]).
\end{aligned}
	\]
Expanding $v_\fh = c^{kj}_j e_k + c^j_{kj} e^k$ we obtain
\[ [v_\fh, e^i] e_i = c^{kj}_j c_{k}^{il} e_le_i - c^{kj}_j c_{kl}^i e^l
e_i + c^j_{kj} c^{ki}_l e^l e_i, \] and similarly 
\[ e^i[v_\fh, e_i] = c^{kj}_j c_{ki}^l e^i e_l + c^j_{kj} c^{k}_{il} e^i
e^l - c^{j}_{kj} c^{kl}_{i}  e^i e_l, \]
from where we obtain
\begin{multline*} [H_\Lambda J] = (2 T + 2\lambda + \chi S) J - \frac{i \lambda}{(k +
h^\vee)^2} \left( \tr|_{\fh^*} \ad ([e^i, e^j])   e_i e_j + \right. \\
\left. \tr|_\fh \ad
([e_i, e_j]) e^i e^j  \right) = (2T + 2\lambda + \chi S) J. 
\end{multline*}
\end{proof}
	\label{sec:drinfelddouble}

\section{Preliminaries on geometry} \label{sec:generalized}
In this section we recall the basic definitions of generalized complex geometry
following \cite{gualtieri1} and \cite{gualtieri2}. We also briefly recall the
notion of \emph{unimodularity} for a Lie algebroid due to Weinstein
\cite{weinstein2}. 

Let $M$ be a smooth manifold and denote by $T$ the tangent bundle of
$M$.
\begin{defn}
	A Courant algebroid is a vector bundle $E$ over $M$, equipped with a
	nondegenerate symmetric bilinear form $\langle ,\rangle $ as
	well as a skew-symmetric bracket $[,]$ on
	$C^\infty(E)$ and with a smooth bundle map $\pi: E
	\rightarrow  T$ called the anchor. This induces a natural
	 differential operator $\mathcal{D}: C^\infty(M) \rightarrow
	 C^\infty(E)$ as $\langle \mathcal{D}f, A\rangle  = \tfrac{1}{2} \pi(A) f$ for
	all $f \in C^\infty (M)$ and $A \in
	C^\infty(E)$. 
 	 These structures should satisfy:
	\begin{enumerate}
		\item $\pi([A,B]) = [\pi(A), \pi(B)], \quad \forall A, B
			\in C^\infty(E)$. 
		\item The bracket $[,]$ should satisfy the following
			analog of the Jacobi identity. If we define the
			\emph{Jacobiator} as $\mathrm{Jac}(A,B,C) =
			[ [A,B], C] + [ [B, C], A ] + [ [ C,A],B]$.  And
			the \emph{Nijenhuis} operator 
			\[ \mathrm{Nij} (A,B,C) = \frac{1}{3} \left( \langle 
			[A,B], C\rangle  + \langle  [B,C], A\rangle  + \langle  [C, A], B\rangle 
			\right). \]
			Then the following must be satisfied:
			\[\mathrm{Jac} (A,B, C) = \mathcal{D} \left(
			\mathrm{Nij} (A, B, C)
			\right), \quad \forall A, B, C \in
			C^\infty(E) \]
		\item $[A, f B] = (\pi(A) f) B + f [A, B] - \langle A, B\rangle 
			\mathcal{D}
			f$, for all $A, B \in C^\infty(E)$ and
			$f \in C^\infty(M)$, 
		\item $\pi \circ \mathcal{D} = 0$, i.e. $\langle \mathcal{D} f,
			\mathcal{D} g\rangle  = 0,
			\quad \forall f,g \in C^\infty(M)$. 
		\item $\pi(A)\langle B, C\rangle  = \langle [A, B] + \mathcal{D}\langle A,B\rangle , C\rangle  + \langle B, [A, C]
			+ \mathcal{D} \langle A,C\rangle \rangle , \quad \forall A, B, C \in
			C^\infty(E)$. 
	\end{enumerate}
	A Courant algebroid $E$ is called \emph{exact} if the following
	sequence is exact:
	\[ 0 \rightarrow T^* \xrightarrow{\pi^*} E \xrightarrow{\pi} T
	\rightarrow 0, \]
	where we use the inner product in $E$ to identify it with its
	dual.
	\label{defn:1}
\end{defn}

This definition extends easily to the complexified situation. 
\begin{ex}
	$E=(T \oplus T^*)\otimes \mathbb{C}$, $\langle,\rangle$ and $[,]$ are respectively the natural symmetric pairing and
	the Courant bracket defined as:
	\begin{equation*}
		\begin{aligned} \langle X + \zeta, Y + \eta\rangle &= \frac{1}{2} \left( i_X \eta + i_Y\zeta
\right).  \\
 {[X + \zeta}, Y + \eta] &= [X, Y] + \mathrm{Lie}_X \eta -
\mathrm{Lie}_Y \zeta - \frac{1}{2} d(i_X \eta - i_Y \zeta).
\end{aligned}
\end{equation*}
\label{ex:courant1}
\end{ex}

From now on, we will work only with exact Courant algebroids, although
some of the results (notably Prop. \ref{prop:universal}) hold in a more
general case.

\begin{defn}[{\cite[4.14]{gualtieri1}}]
	A generalized almost complex structure on a real $2n$-dimensional manifold
	$M$ is given by the following equivalent data:
	\begin{itemize}
		\item an endomorphism $\mathcal{J}$ of $E \simeq T \oplus T^*$ which is
			orthogonal with respect to the inner product
			$\langle, \rangle$. 
		\item a maximal isotropic sub-bundle $L < E \otimes
			\mathbb{C}$ of real index zero, i.e. $L \cap \bar{L} = 0$. 
		\item a pure spinor line sub-bundle $U < \bigwedge^* T^* \otimes
			\mathbb{C}$, called the \emph{canonical line bundle}
			satisfyinng $(\varphi, \bar{\varphi}) \neq 0$ at each point
			$x \in M$ for any generator $\varphi \in U_x$. Here $(, )$
			is the natural inner product induced from $\langle,
			\rangle$. 
	\end{itemize}
	\label{defn:almost-complex}
\end{defn}
The fact that $L$ is of real index zero implies 
\begin{equation*}
	E\otimes \mathbb{C} \simeq (T \oplus T^*) \otimes \mathbb{C} = L \oplus \bar{L} = L \oplus L^*, 
\end{equation*}
using $\langle, \rangle$ to identify $\bar{L}$ with $L^*$.
\begin{defn}[{\cite[4.18]{gualtieri1}}]
	A generalized almost complex structure $\mathcal{J}$ is said to be
	integrable to a generalized complex structure when its $+i$-eigenvalue $L <
	E \otimes \mathbb{C}$ is Courant involutive. 
	\label{defn:complex}
\end{defn}

In this case, $L$ is a Lie bi-algebroid, and $E\otimes
\mathbb{C}$ could be viewed as its \emph{Drinfeld double}. Note that $E$ acts on the sheaf of differential forms
$\bigwedge^\bullet T^*$ via the spinor representation, and this sheaf acquires 
a different grading by the eigenvalues of $\mathcal{J}$ acting via the
spinor representation: \[\bigwedge T^* = U_{-n} \oplus \dots \oplus
U_{n}. \] Clifford multiplication by sections of $\overline{L}$
(resp. $L$)
increases (resp. decreases) the grading. $U_{-n}= U_{\mathcal{J}}$ is the
canonical bundle of $(M, \mathcal{J})$. Given a non-vanishing global
section of $U_\mathcal{J}$, we obtain an isomorphism of sheaves:
\begin{equation}
	\wedge^k \overline{L} \simeq U_{k-n}.
	\label{eq:isomalgebroid}
\end{equation}
The de Rham differential can be
split as $d = \partial + \bar\partial$ such that $\partial: U_k
\rightarrow U_{k-1}$ and $\bar\partial :U_k \rightarrow U_{k+1}$.
\begin{defn}
	A generalized complex manifold $(M, \mathcal{J})$ is called
	generalized Calabi-Yau if the bundle $U_\mathcal{J}$ is
	\emph{holomorphically trivial}, i.e. it admits a non-vanishing closed global
	section.
	\label{defn:gcy}
\end{defn}
In this case the isomorphism (\ref{eq:isomalgebroid}) allows us to
identify the complex $(U_\bullet, \bar{\partial})$ with the complex
computing the Lie algebroid cohomology of $L$. Recall that given any Lie
algebroid $L$ over $M$, we can define a differential $d_L: C^{\infty}(M)
\rightarrow C^{\infty}(L^*)$ as $(d_L f)(l) = \pi_L (l) f$, where $l$ is
a section of $L$ and $\pi_L$ is the anchor map of $L$. This differential
can be extended to $\bigwedge^\bullet L^*$ by imposing the Leibniz rule
in the usual way (for $\zeta \in C^{\infty} (\bigwedge^{k-1} L^*)$):
\begin{multline}
	(d_L \zeta)(l_1, \dots, l_{k}) = \sum_i (-1)^{i+1} \pi(l_i)
	\zeta(l_1, \dots, \hat{l_i}, \dots l_k) + \\\sum_{i<j} (-1)^{i+j}
	\zeta([l_i, l_j], \dots, \hat{l_i}, \dots, \hat{l_j}, \dots,
	l_k).
	\label{eq:diffeerentialmult}
\end{multline}
The cohomologies of the complex $(\bigwedge^\bullet L^*, d_L)$ are
denoted by $H^\bullet (L)$ and are called the Lie algebroid cohomologies
of $L$ (with trivial coefficients).
If $(M, \mathcal{J})$ is a generalized Calabi-Yau manifold. The
isomorphism of (\ref{eq:isomalgebroid}) is an isomorphism of complexes
(using $\bar{L} = L^*$).
Moreover, in this case, the Lie algebroids $L$ and $L^*$ are both
\emph{unimodular}, a notion that we now recall. 

For a Lie algebroid $L$ we have the corresponding sheaf of twisted
differential operators $U(L)$. The sheaf $\bigwedge^\mathrm{top} T^*$ is
always a right twisted D-module, and the corresponding left $U(L)$-module
is then the line bundle $Q_L = \bigwedge^{\mathrm{top}} L \otimes
\bigwedge^{\mathrm{top}} T^*$. Suppose for simplicity that the line bundle $Q_L$ is trivial,
for each non-vanishing section $s$ of $Q_L$ we can define $\theta_s \in
C^{\infty} L^*$ by \[ \theta_s(l) s = l \cdot s, \] where we use the left
D-module structure of $Q_L$ on the RHS.  It turns out that $\theta$ gives
rise to a well 
defined element of $H^1 (L, Q_L)$, the first Lie algebroid
cohomology of $L$ with coefficients in $Q_L$ (see \cite{weinstein1} for
details).
\begin{defn}[\cite{weinstein2}]
	A Lie algebroid $L$ is called \emph{unimodular} if the class
	$\theta \in H^1(L, Q_L)$ above constructed vanishes.
	\label{defn:unimodularity}
\end{defn}
Now let $L$ be a unimodular Lie algebroid (of rank $k$). Let $U$ be open
in $M$ and choose a local
frame $\{e_i\}$ for
$L$, we obtain a section $s = e_1 \wedge \dots \wedge e_k$ for
$\bigwedge^\mathrm{top}L$. We can choose a local volume form $\mu \in
C^\infty (\bigwedge^{top} T^*)$ such that the class $\theta$ is
represented by zero (we may need to shrink $U$) . If we define the \emph{structure constants} of $L$
by $[e_i, e_j]= c_{ij}^k e_k$ we obtain the identity:
\begin{equation}
	\dive_\mu e_k = -c_{ki}^i, 
	\label{eq:divergence}
\end{equation}
where we sum over repeated indexes and
$\dive_\mu e_k$ is the divergence of $e_k$ with respect to the
volume form $\mu$, defined by \[(\dive_\mu e_k) \mu = \lie_{\pi_L(e_k)} \mu.
\]
We have the following Proposition (see for example \cite[Theorem
10]{zabzine4})
\begin{prop}
	A generalized complex manifold $M$ is generalized Calabi-Yau if
	and only if $U_\mathcal{J}$ is trivial and $L$ is unimodular.
	\label{prop:unimodularity}
\end{prop}
Recall from \cite[Prop. 2.2.2]{gualtieri1} that we have an isomorphism
\begin{equation}
	U_\mathcal{J} \otimes U_\mathcal{J} \simeq \det L \otimes \det
	T^*.
	\label{eq:isomu}
\end{equation}
Therefore given a generalized Calabi-Yau manifold $M$ and a local frame
$\{e_i\}$ for $L$, we can choose a closed pure spinor such that the
corresponding volume form $\mu$ satisfies (\ref{eq:divergence}).

Given a generalized complex manifold $(M, \mathcal{J})$ the projection $F =
\pi_T(L)$ gives rise to a smooth integrable distribution $\Delta$ defined by
$\Delta \otimes \mathbb{C} = F \cap \bar{F}$. 
\begin{prop}[\cite{gualtieri1}]
	Let $(M, \mathcal{J})$ be a generalized complex manifold, and let $x$ be a
	point in $M$ such that $\dim \Delta$ is locally constant at $x$. Then there
	exists an open neighborhood of $x$ in $M$ which is expressed as a product
	of a complex manifold times a symplectic one.
	\label{prop:product}
\end{prop}
We remark however that there exist generalized Complex and generalized Calabi-Yau manifolds with
points where $\dim \Delta$ is not locally constant.

\section{Sheaves of vertex algebras} \label{sec:sheaf}
In this section we recall some results from \cite{gerbes2} and \cite{bressler1}
in the language of SUSY vertex algebras, following \cite{heluani8}.

Let $(E, \langle ,\rangle , [,], \pi)$ be a Courant algebroid. 
Let  $\Pi E$ be the corresponding purely odd super vector bundle. we will abuse
notation and denote by $\langle, \rangle$ the corresponding
super-skew-symmetric bilinear form, and by $[,]$ the corresponding super-skew-symmetric
degree $1$ bracket on $\Pi E$. Similarly, we obtain an odd
differential operator $\mathcal{D}: C^\infty(M) \rightarrow 
C^\infty(\Pi E)$. If no confusion should arise, when $v$ is an element of a vector
space $V$, we will denote by the same symbol $v$ the corresponding element of $\Pi
V$, where $\Pi$ is the \emph{parity change operator}. Recall that for
sections of $E$ we have the Dorfman bracket $\circ$ which is defined in
terms of the Courant bracket and $\mathcal{D}$ as 
\begin{equation}
	X \circ Y = [X, Y] +
\mathcal{D} \langle X, Y \rangle.
\label{eq:dorf}
\end{equation}

The following proposition from from \cite{heluani8} (cf.
 \cite{bressler1}) describes the construction of the chiral de Rham
 complex in parallel to the construction of \emph{twisted differential
 operators} given a Lie algebroid: 
 \begin{prop}
	 For each complex Courant algebroid $E$ over a differentiable manifold
	 $M$, there exists a sheaf $U^\mathrm{ch}(E)$ of SUSY vertex algebras on $M$
	 generated by functions $i: C(M) \hookrightarrow
	 U^\mathrm{ch}(E)$, and sections of $\Pi E$, $j: C( \Pi E)
	 \hookrightarrow U^\mathrm{ch}(E)$ subject to the relations:
	 \begin{enumerate}
		 \item $i$ is an ``embedding of algebras'', i.e. 
			 $i(1) = \vac$, and $i(fg) =
			 i(f) \cdot i(g)$, where in the RHS we use the normally
			 ordered product in $U^\mathrm{ch}(E)$.
		 \item $j$ imposes a compatibility condition between the
			 Dorfman bracket in $E$ and the Lambda bracket in
			 $U^\mathrm{ch}(E)$: \[ [j(X)_\Lambda j(Y)] =
			 j (X\circ Y) + 2\chi i (\langle X, Y \rangle).\]
		 \item $i$ and $j$ preserve the $\cO$-module structure of
			 $E$, i.e. $j (f X) = i(f)
			 \cdot j(X)$. 
		 \item $\mathcal D$ and $S$ are compatible, i.e. $j
			 \mathcal D f = S i (f)$. 
		 \item We impose the usual commutation relation \[
			 [j(X)_\Lambda i(f)] = i (\pi(X) f).\]
	 \end{enumerate}
	 In the particular case when $E = (T\oplus T^*) \otimes
	 \mathbb{C}$ is the standard
	 Courant algebroid, then $U^\mathrm{ch}(E)$ is the chiral de Rham
	 complex of $M$, denoted by $\Omega^\mathrm{ch}_M$ for historical
	 reasons.
	 \label{prop:universal}
 \end{prop}
Using this proposition, we will abuse notation and use the same symbols for
sections of $E\otimes \mathbb{C}$ when they are viewed as sections of
$U^\mathrm{ch}(E)$. 
\section{$N=2$ supersymmetry} \label{sec:sections}
In this section we state the main results of this article, we postpone their proofs
for the Appendix. We also show that these supersymmetries generalize those of
\cite{heluani2} and \cite{heluani8} in the Calabi-Yau and symplectic
case. 

Let $M$ be an orientable and differentiable manifold of Real dimension $N$. Let $T :=
T_M$ be its tangent bundle and $T^* := T^*_M$ its cotangent bundle, and
let $E$ be an exact Courant algebroid on $M$.   Let
$\mathcal{J} \in \End E \simeq \End (T \oplus T^*)$ be a generalized complex structure
and let $L \leq (T \oplus T^*)\otimes \mathbb{C}$ be the corresponding
Dirac structure,  and let $\rho:= {\rho^i}_j$ be the transition functions
for $L$. Let $\{e_i\}$ be a local frame for $L$ and $\{e^i\}$ be the
corresponding dual frame for $L^* \simeq \bar{L}$, i.e. $\langle e^i,
e_j\rangle = {\delta^i}_j$. 

Let $U := U_L \leq \bigwedge^* T^*$ be the associated
canonical bundle. It follows from (\ref{eq:isomu}) that if $U$ is
trivial, then $\det L$ and $\det L^*$ are trivial bundles. Moreover, we
will fix a closed pure spinor and a corresponding volume form $\mu$ such that
(\ref{eq:divergence}) holds for the frame $\{e_i\}$ and the
corresponding dual statement holds for the frame $\{e^i\}$.
\begin{lem}
	Let $M$ be a generalized Calabi-Yau manifold.  Then the
	following defines a global section of $U^\mathrm{ch}(E)$: (we sum
	over repeated indexes)
	\begin{equation}
		J = \frac{\sqrt{-1}}{2} e^ie_i. 
		\label{eq:current1}
	\end{equation}
	\label{lem:section1}
\end{lem}
\begin{proof}
	It follows from Proposition \ref{prop:universal} that  under a
	change of coordinates, $J$ 
	transforms as:
	\begin{equation}
		\frac{i}{2} \left(  {(\rho^*)^i}_j e^j\right) \left( {\rho_i}^k
		e_k
		\right).
		\label{eq:1.1}
	\end{equation}
	To apply quasi-associativity we need to compute the $\chi$ terms
	in the Lambda bracket $[{e^j}_\Lambda {\rho_i}^k e_k]$ and these
		are easily shown to be $2 \chi {\rho_i}^k {\delta^j}_k$.
		Therefore (\ref{eq:1.1}) reads:
		\begin{multline}
			\frac{i}{2}{(\rho^*)^i}_j \Bigl[ e^j \bigl( {\rho_i}^k e_k \bigr)
			\Bigr] + i T\left( { (\rho^*)^i}_j \right) {\rho_i}^j
			=\frac{i}{2}  {(\rho^*)^i}_j \Bigl[ e^j \bigl( e_k {\rho_i}^k \bigr)
			\Bigr] + \\ + i T\left( { (\rho^*)^i}_j \right)
			{\rho_i}^j
			=\frac{i}{2} {(\rho^*)^i}_j \Bigl[ \bigl(e^j e_k\bigr) {\rho_i}^k
			\Bigr] + i T\left( { (\rho^*)^i}_j \right) {\rho_i}^j
			= 
			\\ = \frac{i}{2} \left[ {(\rho^*)^i}_j {\rho_i}^k \right]
			\left(e^j e_k
			\right) + i T \left( {(\rho^*)^i}_j \right) {\rho_i}^j =
			\frac{i}{2} e^i e_i + i T ({(\rho^*)i}_j) {\rho_i}^j.
			\label{eq:2.1}
		\end{multline}
			Using 
			\begin{equation}
				\frac{\partial \det \rho}{\partial x_a} =
				\frac{\partial \det \rho}{\partial {\rho_i}^j}
				\frac{\partial {\rho_i}^j}{ \partial x_a} = (\det
				\rho) {(\rho^{-1})_j}^i \frac{\partial
				{\rho_i}^j}{\partial x_a},
				\label{eq:det_eq}
			\end{equation}
			we see that this becomes: 
			\[\frac{i}{2} e^i e_i + i \det \rho \,T \det
			\rho^{-1}. \] The second term of this last
			expression can be chosen to be zero if $c_1(L) =
			0$, which in turn happens if $U$ is trivial.
		\end{proof}
	\begin{rem}
		Note that in order to construct these sections, in
		\cite{heluani2} and \cite{heluani8} the authors used a \emph{connection}
		on $T_M$. This is replaced in this setting with the existence of a global
		section of $U$.
		\label{rem:noconnection}
	\end{rem}
	\begin{lem} Let $M$ be a generalized Calabi-Yau manifold. 
		Define a local section of $U^\mathrm{ch}(E)$ by:
		\begin{multline}
		H :=  \frac{1}{4}  \Bigl[ e^i \bigl( e^j[e_i,e_j] \bigr) +
		e_i \bigl( e_j [e^i, e^j] \bigr) \Bigr] - \frac{i}{2} T \mathcal{J} [e^i,
		e_i]  +
		 \frac{1}{2} \left( e_j Se^j
		+ e^j Se_j.
		\right) 
		\label{eq:h.def}
		\end{multline}
	 Then the following is true:
		\begin{enumerate}
			\item $H$ defines a global section of
				$U^\mathrm{ch}(E)$.
			\item We have the following OPE:
				\begin{equation*}
					{[J}_\Lambda J] = - \left( H + \frac{c}{3} \lambda
					\chi \right), \qquad c = 3 \dim_\mathbb{R} M.
				\end{equation*}
		\end{enumerate}
	\label{lem:h.def}
	\end{lem}
	\begin{proof}
		The proof can be found in the appendix.
	\end{proof}
	\begin{rem}
		Note that the field $H$ has the form (\ref{eq:hdefsind})
		with $k + h^\vee = 2$. We therefore may view this
		construction as an algebroid generalization of the
		Kac-Todorov or super-Sugawara construction.
		\label{rem:longshot}
	\end{rem}
		\begin{thm}
		Let $M$ be a generalized Calabi-Yau manifold.
		Let $J$ and $H$ be the corresponding sections of 
		$U^\mathrm{ch}(E)$ as constructed in Lemmas \ref{lem:section1}
		and \ref{lem:h.def}. 
		\begin{enumerate}
			\item For a function $f \in C^\infty (M)$, the
				corresponding field of $U^\mathrm{ch}(E)$ is
				primary of conformal weight $0$ with
				respect to $H$, namely:
				\[ [H_\Lambda f] = (2T + \chi S) f. \]
			\item For a section $X \in C^\infty (E)$, the
				corresponding field of $U^\mathrm{ch}(E)$
				has conformal weight $1/2$ with respect
				to $H$, but it is \emph{not primary}, it
				satisfies:
				\begin{equation}
					[H_\Lambda X] = (2T + \lambda +
					\chi S) X + \lambda \chi
					\dive_\mu X.
					\label{eq:notprimary}
				\end{equation}
			\item The fields $H$ and $J$ generate an $N=2$
				SUSY
				vertex algebra of central charge $c = 3
				\dim_\mathbb{R} M$.
		\end{enumerate}
		\label{thm:n=2thm}
	\end{thm}
	\begin{proof}
		The proof can be found in the Appendix. 
	\end{proof}
\begin{defn} 
		Given a generalized Calabi-Yau manifold $M$, we will say
		that $M$ \emph{has a nice (local) volume form} if we can find a
		(local)
		volume form $\mu$ and (local) dual frames $\{e_i\}$, $\{e^i\}$
		for $L$ and $L^*$ such that \[ \dive_\mu e_i = \dive_\mu
		e^i = 0 \qquad \forall i.\] Using 
		Prop. \ref{prop:unimodularity}, this is equivalent to finding local
		frames such that $ \sum_i [e^i, e_i] = 0. $
		\label{defn:conditionstrange}
	\end{defn}
	\begin{rem} \hfill
		\begin{enumerate}
			\item
	 	Calabi-Yau manifolds and symplectic manifolds admit nice
		volume forms. Moreover, if the generalized complex
		structure is regular (i.e. there is no type jump), then
		$M$ admits a nice volume form. Indeed, using Prop.
		\ref{prop:product} one can find local frames $\{e_i\}$
		and $\{e^i\}$ such that \emph{all structure constant
		vanish.} In this case, the first two terms in the  field
		$H$ in Lemma
		\ref{lem:h.def} vanish.
		We do not know if every generalized Calabi-Yau manifold
		admits a nice local volume form. 

		On the other hand, there are examples\footnote{We owe A.
		Weinstein for an explanation of this point.} of unimodular
		Lie algebroids not admiting local frames
		$\{e_i\}$ with $\dive_\mu e_i = 0 \, \forall i$. 
		
	\item The proof of Theorem \ref{thm:n=2thm} is much simpler in the
		regular case as $H$ would then be quadratic  (see Remark
		\ref{rem:nosabemos}). Moreover, if $M$ admits a nice
		volume form, the proof of this Theorem, while still a
		long computation, would be much simpler.

		In the general case we have to make explicit use of
		unimodularity for $L$ and $L^*$ to find a spinor and
		volume form such that (\ref{eq:divergence}) holds.
	\item It follows from (\ref{eq:notprimary}) that if the manifold
		$M$ admits a local nice volume form, we can find a local
		frame for $E\simeq T\oplus T^*$ consisting of primary
		fields.
	\end{enumerate}
		\label{rem:nosabemos}
	\end{rem}
	\begin{ex}
		Let $M$ be a complex manifold, with complex structure $J$, then we can
		consider the generalized complex structure:
		\begin{equation*}
			\mathcal{J}=  \begin{pmatrix} J & 0 \\ 0 & -J^* \end{pmatrix}.
		\end{equation*}
		The corresponding line bundle $U$ is the canonical bundle of $M$, thus we
		are in the usual Calabi-Yau case. $L = T_{1,0} \oplus T^*_{0,1}$ and choosing
		holomorphic coordinates we can take $e_\alpha = \partial_{z^\alpha}$,
		$e_{\bar\alpha} = d z^{\bar{\alpha}}$. Similarly $e^\alpha = 2 d
		z^\alpha$, $e^{\bar\alpha} = 2 \partial_{z^{\bar{\alpha}}}$.
		Finally, we have a global holomorphic volume form $\Omega$, which satisfies: 
		\begin{equation*}
			\Omega \wedge \overline{\Omega} = \sqrt{\det g} dz^1 \dots dz^N d
			\bar{z}^1 \dots d\bar{z}^N,
		\end{equation*}
		where $g = g_{ij}$ is the K\"ahler metric on $M$. We can
		pick coordinates where the volume form $\Omega \wedge
		\overline{\Omega}$ is constant. In this frame, the fields
		$J$ and $H$ are (note that all Courant brackets vanish):
		\begin{equation*}
			\begin{aligned}
				J &= i dz^\alpha \partial_{z^\alpha}
				- i d z^{\bar\alpha}
				\partial_{z^{\bar{\alpha}}},\\
				H &= \partial_{z^\alpha} T z^\alpha  + d z^{\bar{\alpha}}
				S\partial_{z^{\bar{\alpha}}} + dz^\alpha S
				\partial_{z^\alpha} + \partial_{z^{\bar{\alpha}}} T
				z^{\bar{\alpha}},
			\end{aligned}
		\end{equation*}
		which are the generators of the $N=2$ superconformal structure of
		\cite{malikov}.
		\label{ex:complex}
	\end{ex}
	\begin{ex}
		In the symplectic case we have a generalized complex structure 
		\begin{equation*}
			\mathcal{J} = \begin{pmatrix} 0 & - \omega^{-1} \\ \omega & 0
			\end{pmatrix},
		\end{equation*}
		where $\omega$ is a symplectic form viewed as a map $T \rightarrow T^*$. We
		can use Darboux' Theorem to find coordinates $x^i, y^i$ such that $\omega$
		takes the standard form:
		\begin{equation*}
			\omega = dx^1 \wedge dy^1 + \dots dx^N \wedge
			dy^N.
		\end{equation*}
		In this coordinate system we can take $\{ \partial_{x^i} - i
		d y^i, dx^i - i \partial_{y^i}\}$ as a frame for $L$ and $\{ dx^i + i
		\partial_{y^i}, \partial_{x^i} + i d y^i\}$ as the dual frame.  The supersymmetry
		generators are now
		\begin{equation*}
			\begin{aligned}
				J &= \partial_{x^i} \partial_{y^i} + dx^i dy^i, \\
				H &= dx^i S \partial_{x^i} + dy^i S \partial_{y^i} + Tx^i
				\partial_{x^i} + T y^i \partial_{y^i}, 
			\end{aligned}
		\end{equation*}
		and these are the generators of \cite[Lem. 5.2.]{heluani8}
		\label{ex:symplectic}
	\end{ex}
In general, from the frames $\{e_i\}$ and $\{e^i\}$, we obtain an
orthonormal frame $\{a^i\}$ for the Courant algebroid $E \otimes
\mathbb{C} \simeq (T \oplus
T^*) \otimes \mathbb{C}$ as $a^i = \tfrac{1}{\sqrt{2}} (e^i + e_i)$ for
$i=1, \dots, N=\dim M$ and
$a^i = \tfrac{1}{\sqrt{-2}} (e^{i-N} - e_{i-N})$ for $i=N+1, \dots, 2N$. If
the manifold $M$ admits a nice volume form, then in this frame, the
section $H$ looks like
\begin{equation}
	H = \frac{1}{2} \sum_{i=1}^N a^i Sa^i + \frac{1}{12}
	\sum_{i,j=1}^N [a^i, a^j] (a^i a^j).
	\label{eq:conjecture}
\end{equation}
For a general orientable manifold $M$ with a Courant algebroid $E$
over it, we do not know if the expression (\ref{eq:conjecture}) defines a
global section of $U^\mathrm{ch}(E)$. 

\subsection{Topological twist} 

As with any vertex algebra with an $N=2$ superconformal structure, some
of the Fourier modes of the superfields $H$ and $J$,
 when expanded as in (\ref{eq:fourier}), play a special role.
 The operator
 $L_0 := \tfrac{1}{2}(H_{(1|0)} + J_{(0|1)})$ acts diagonally on $V$ and its
eigenvalues are called the conformal weights of the corresponding states.
These conformal weights are different that the ones we were considering
before, and they correspond to the twisted theory.

Similarly the eigenvalues of $J_0 :=  -i J_{(0|1)}$ are called charge.
Sections of $L$ have conformal weight $1$ and charge $-1$, sections of
$L^*$ have conformal weight $0$ and charge $1$ (this follows for example
from (\ref{eq:a.1.2}) and (\ref{eq:a.1.5}) below), while functions
have conformal weight $0$ and charge $0$. It follows easily that the conformal weight zero part
of $U^\mathrm{ch}(E)$ is just $\bigwedge^\bullet L^*$ and the natural
grading on
$\bigwedge^\bullet L^*$ coincides with the charge gradation on
$U^\mathrm{ch}(E)$.  

There are also two more important Fourier modes: the BRST charge $Q_0$, and the
homotopy operator $G_0$ (we maintain here the notation of
\cite{malikov}). In this context $Q_0 := \tfrac{1}{2}(H_{(0|1)} + i
J_{(0|0)})$ and $G_0 :=  \tfrac{1}{2} (H_{(0|1)} - i J_{(0|0)})$ (these
correspond to the zero modes of the fields $G^\pm$ appearing in
(\ref{eq:decomposeas})). The operator $Q_0$ increases the charge and
squares to zero, and we consider the complex $(U^\mathrm{ch}(E), Q_0)$. 
The operator $G_0$ decreases the charge and also
squares to zero.  In a similar way as in \cite{malikov} we have:
\begin{thm}
	For a generalized Calabi-Yau manifold we have a quasi-isomorphism
	of complexes
	\begin{equation}
		(U_\bullet, \bar\partial) \simeq (\wedge^\bullet L^*,
		d_L)
		\hookrightarrow (U^\mathrm{ch}(E), Q_0)
		\label{eq:quasi-iso}
	\end{equation}
	\label{thm:quasi-iso}
\end{thm}
\begin{proof}
	We have $[G_0, Q_0] = L_0$, hence the operator
$G_0$ is an homotopy to zero for the operator $Q_0$ away from conformal
weight $0$. We need to show that $Q_0$ restricted to $\bigwedge^\bullet
L^*$ acts as $d_L$. The question is local in $M$, therefore we may
restrict ourselves to a small neighborhood with dual frames $\{e_i\}$ and
$\{e^i\}$ for $L$ and $L^*$. $Q_0$ being a zero mode of a field is a
derivation of all products in the vertex algebra, therefore it suffices
to check that it acts as $d_L$ on functions and the sections $e^i$. This
last statement is a straightforward computation.
\end{proof}
\begin{rem}\hfill
	\begin{enumerate}
	\item It is well known that the BRST cohomology of a topological
		vertex algebra carries the structure of a
		Gerstenhaber algebra \cite{lian2} (see also \cite{ima1}
		and references therein). In this case
		we see that we recover the Gerstenhaber algebra structure in the
		Lie algebroid cohomology of a generalized Calabi-Yau
		manifold. Indeed the sheaf $U^\mathrm{ch}(E)$ itself can
		be viewed as a sheaf of $G_\infty$-algebras \cite{ima1}.
		\item
	Just as in the usual case, the zero-th Fourier modes above defined make
	sense on any generalized complex manifold, even if the
	superfields $H$ and $J$ are not well defined. We will not pursue
	this further in this article.
\item It is not obvious how to diagonalize $J_0$ on
	$U^\mathrm{ch}(E)$, for example, the local sections \[ J_i = Se_i
	+ \frac{1}{4} c_i^{jk} e_j e_k, \]  have charge zero and
	conformal weight $1$. The BRST charge is the residue of
	the field \[ Q = e^i J_i + \frac{1}{4} e^i (e^j [e_i, e_j]),\]
	and note the similarity of this field with the Chevalley
	differential computing Lie algebra cohomology.
\end{enumerate} 
	\label{rem:pursue}
\end{rem}
The notion of conformal weight does not make sense for
	differential forms nor vector fields. Instead, one has to
	consider the mixed sections $e^i$. This is particularly
	interesting when the generalized Calabi-Yau manifold has
	\emph{type jumps}. 
 	If one considers a Calabi-Yau manifold as in example
	\ref{ex:complex} then, performing the topological twist mentioned
	makes holomorphic forms and antiholomorphic vector fields have
	conformal weight $0$ (the B-model). In the symplectic case
	(A-model) however, for a
	Darboux local system of coordinates as in Example
	\ref{ex:symplectic}  we obtain that the basic fermions of conformal weight
	zero are of the form $dx^i + i \partial_{y^i}$ and
	$\partial_{x^i} + i dy^i$. 
	\section{Concluding remarks}
In this article we produced an embedding of the $N=2$ super vertex
algebra of central charge $c = 3 \dim M$ into the chiral de Rham complex of any generalized Calabi-Yau
manifold $M$. Our approach works without modification in the
\emph{twisted} generalized Calabi-Yau case (i.e. when the exact Courant
algebroid $E$ is not the standard one). 
We discussed the topological twist of this $N=2$ algebra, leading to an
identification of the BRST-cohomology of the chiral de Rham complex of
$M$ with the Lie-algebroid cohomology of the associated Dirac structure.

The formulae for the generators of $N=2$ supersymmetry look like the
generators of \cite{getzler1}. In particular the generators for $N=1$
supersymmetry look like the generators of the Kac-Todorov construction.
Thus, we can view our results as a Courant algebroid generalization of
the Kac-Todorov construction and the results in \cite{getzler1}.  

Similar results hold in the Generalized K\"ahler and Generalized
Calabi-Yau metric cases. In particular it is possible to show that in the
latter case there are two commuting sets of $N=2$ structures. Taking
BRST cohomology with respect to the \emph{left charge} we obtain a sheaf
with finite dimensional cohomologies (in each conformal weight) and with a surviving topological
structure (given by the \emph{right $N=2$ structure}). This allows us to
define the \emph{Elliptic genus} of a generalized Calabi-Yau metric
manifold just as in the usual case \cite{borisovlibgober}. We plan to
return to this matters in the future.

	\appendix

	\section{Proofs of the main results}

	\begin{proof}[Proof of Lemma \ref{lem:h.def}]
		1) follows from 2). The proof of (2) is similar to the
		proof of Prop \ref{prop:doubles}, the major difficulty is
		that the ``structure constants'' defined  by
		\[
		\begin{aligned}
			c^{i}_{jk} &:= \langle e^i, [e_j, e_k] \rangle =
			\langle [e^i, e_j], e_k \rangle \in C(M), \\ 
			c^{ij}_k &:= \langle [e^i, e^j], e_k \rangle =
			\langle e^i, [e^j, e_k] \rangle \in C(M),
		\end{aligned}
		\] are functions, therefore we need to keep track of
		quasi-associativity terms. Note that we have also used
		axiom (5) in Def. \ref{defn:1}.
	In order to compute $[J_\Lambda J]$ we start with:
	\begin{equation}
		\begin{aligned}
			{[e_j}_\Lambda J] &= \frac{i}{2} \left({[e_j}, e^i] + 2
			\chi {\delta_j}^i  \right) e_i + \frac{i}{2} e^i
			{[e_j},e_i]  +
			i \int_0^\Lambda  \eta \, c^i_{ij} d\Gamma, \\ &=
			\frac{i}{2} \left( (c_j^{ik} e_k) e_i - (c_{jk}^i
			e^k) e_i) + e^i (c_{ji}^k e_k)\right)
			 +i \chi e_j + i 
			 \lambda  c_{ij}^i, \\ &=
			\frac{i}{2} \left( c_j^{ik} (e_k e_i) - c_{jk}^i
			(e^k e_i) + c_{ji}^k (e^ie_k)\right)
			 +i \chi e_j + i
			 \lambda  c_{ij}^i - i Tc^{i}_{ji}, \\ &= \frac{i}{2} c^{ik}_j
			 (e_k e_i) + i \chi e_j + i (\lambda + T)
			 c^{i}_{ij}.
		\end{aligned}
		\label{eq:a.1.1}
	\end{equation}
	 By skew-symmetry we obtain:
	\begin{equation}
			{[J}_\Lambda e_j] = \frac{i}{2} c^{ik}_j (e_k
			e_i) -i (\chi + S) e_j - i
			 \lambda   c^{i}_{ij}.
		\label{eq:a.1.2}
	\end{equation}
	Similarly we have
	\begin{equation}
		\begin{aligned}
			{[e^j}_\Lambda J] &= \frac{i}{2} {[e^j}, e^i]
			 e_i + \frac{i}{2} e^i
			 \left( {[e^j},e_i] + 2 \chi {\delta^j}_i \right) +  
			i \int_0^\Lambda  \eta \, c^{ji}_i d\Gamma, \\ &=
			\frac{i}{2} \left( (c^{ji}_k e^k) e_i +
			e^i (c^{j}_{ik}  e^k) - e^i (c^{jk}_i  e_k)\right)
			-   i \chi e^j + i
			\lambda c^{ji}_i,  \\ &=
			\frac{i}{2} \left( c^{ji}_k (e^k e_i) +
			c^{j}_{ik} (e^i e^k) - c^{jk}_i (e^i e_k)\right)
			-   i \chi e^j + i
			(\lambda + T) c^{ji}_i,  \\ &= \frac{i}{2} c^j_{ik} (e^i
			e^k) - i \chi e^j + i (\lambda + T) c^{ji}_i,
		\end{aligned}
		\label{eq:a.1.4}
	\end{equation}
	and by skew-symmetry we obtain:
	\begin{equation}
			{[J}_\Lambda e^j] = \frac{i}{2} c^j_{ik} (e^i
			e^k) + i (\chi + S) e^j - i
			 \lambda   c^{ji}_i.
		\label{eq:a.1.5}
	\end{equation}
	From (\ref{eq:a.1.2}) and (\ref{eq:a.1.5}) we obtain:
	\begin{multline}
		{[J}_\Lambda J] = -\frac{1}{4} \left(
		c^j_{ik} (e^i e^k) + 2 (\chi + S)
		e^j - 2 \lambda  c^{ji}_i \right) e_j + \\
		\frac{1}{4} e^j \left(
	c^{ik}_j (e_k e_i) - 2 (\chi +S) e_j - 2 \lambda 
	c^i_{ij} \right) - \frac{1}{4}  \int_0^\Lambda [c^j_{ik} (e^i
	e^k)_\Gamma e_j]
	d\Gamma - \\ - \frac{1}{2} \int_0^\Lambda (\eta - \chi) ([e^j,
	e_j] + 2 \eta \dim M ) d\Gamma.
	\label{eq:complicada}
\end{multline}
We can compute the integral term easily as:
\begin{equation*}
	\lambda  c^i_{ij} e^j  - \frac{1}{2} \lambda  [e^j,
	 e_j] - \lambda \chi \dim M,
	\end{equation*}
and replacing in (\ref{eq:complicada}) using quasi-associativity we obtain:
\begin{multline}
	{[J}_\Lambda J] =  -\frac{1}{4} \bigl( e^i (e^k [e_i, e_k])
	+ e_i (e_k [e^i, e^k]) \bigr) - \frac{1}{2} \bigl(e_j Se^j + e^j Se_j
	\bigr) \\  - \frac{1}{2} (\lambda + T) [e^j, e_j]  + 
	\frac{1}{2} \lambda (c^{ji}_i e_j + c^i_{ij} e^j) -
	\lambda \chi \dim M  + T (c^i_{ij} e^j) \\ = -\frac{1}{4} \bigl( e^i (e^k [e_i, e_k])
	+ e_i (e_k [e^i, e^k]) \bigr) - \frac{1}{2} \bigl(e_j Se^j + e^j Se_j
	\bigr) + \\ + \frac{1}{2} T( c^i_{ij} e^j + c^{ij}_i e_j ) -
	\lambda \chi \dim M
\end{multline}

	\end{proof}
	\begin{proof}[Proof of Theorem \ref{thm:n=2thm}] 1) is a straighforward computation, we leave it as
		an excercise for the reader. We first prove 2) for the
		sections $e_k$.
		For this we need:
		\begin{multline}
			{[e_k}_\Lambda e^i Se_i + e_i Se^i] = (c_{k}^{ij}
			e_j - c_{kj}^i e^j)Se_i + 2 \chi Se_k + (c_{ki}^j
			e_j) Se^i + e^i (S + \chi) (c^j_{ki} e_j) + \\ e_i
			(S + \chi) (c_{k}^{ij} e_j - c_{kj}^i e^j + 2
			\delta_k^i \chi) + \\ \int_0^\Lambda [{c_{k}^{ij}
			e_j - c_{kj}^i e^j}_\Gamma Se_i]  d\Gamma +
			\int_0^\Lambda [{c^j_{ki} e_j}_\Gamma Se^i]
			d\Gamma.
			\label{eq:cuad1}
		\end{multline}
		The constant term in (\ref{eq:cuad1}) is given by
		\begin{multline}
			(c_{k}^{ij}
			e_j - c_{kj}^i e^j)Se_i + (c_{ki}^j e_j) Se^i +
			e^i (c^j_{ki} Se_j) + e_i (c^{ij}_k Se_j) - e_i
			(c^{i}_{kj} Se^j) + \\ e^i \left(  (Sc^j_{ki}) e_j
			\right) + e_i \left( (S c^{ij}_k) e_j -
			(Sc^{i}_{kj}) e^j\right) = (c^{ij}_k e_j -
			c^{i}_{kj} e^j + c^{ji}_k e_j+ c^{i}_{kj} e^j) Se_i + \\ (T e_i)
			c^{ij}_{k,j} - (T c^{ij}_k) (c_{ij}^l e_l) +
			(Te^i) c^{j}_{ki, j} - (T c^j_{ki}) c^i_{jl}e^l +
			(T c^j_{ki}) c^{il}_j e_l + \\
			(c^j_{ki} e_j - c^{j}_{ki} e_j)
			Se^i - (T e_i) c^{i,j}_{kj} + (Tc^{i}_{kj})
			c_{i}^{jl} e_l - (Tc^{i}_{kj}) c_{il}^j e^l + \\e^i \left(  (Sc^j_{ki}) e_j
			\right) + e_i \left( (S c^{ij}_k) e_j -
			(Sc^{i}_{kj}) e^j\right) = (Te_i) (c^{ij}_{k,j} -
			c^{i,j}_{kj} ) + (Te^i) c^j_{ki,j} + \\ \left(2 (Tc^j_{ki})
			c^{il}_j - (Tc^{ij}_k) c^l_{ij} \right) e_l - 2
			(T c^{j}_{ki}) c^{i}_{jl} e^l +\\
			\frac{1}{2} e^i \left( (c^j_{ki,l} e^l) e_j
			+
			(c^{j,l}_{ki} e_l) e_j\right) +
			\frac{1}{2} e_i \left(
			(c^{ij}_{k,l} e^l) e_j + (c^{ij,l}_k e_l) e_j -
			(c^{i,l}_{kj} e_l) e^j - (c^{i}_{kj,l} e^l) e^j
			\right) = \\ (Te_i) (c^{ij}_{k,j} -
			c^{i,j}_{kj} ) + (Te^i) c^j_{ki,j} +  \left(2 (Tc^j_{ki})
			c^{il}_j - (Tc^{ij}_k) c^l_{ij} \right) e_l - 2
			(T c^{j}_{ki}) c^{i}_{jl} e^l +\\
			+ \frac{1}{2} c^{ij,l}_k e_i e_l e_j +
			\frac{1}{2} c^{j}_{ki, l} (e^i (e^l e_j)) +
			e^i T(c^j_{ki,j}) - \frac{1}{2} c_{kj,l}^i (e^l
			(e^j e_i)) +\frac{1}{2} c^{j,l}_{ki} (e^i (e_l
			e_j)) - \\ \frac{1}{2} c^{ij}_{k,l} (e^l(e_ie_j))
			+ e_i (Tc^{ij}_{k,j}) 
			 -
			\frac{1}{2} c^{i,l}_{kj} (e^j (e_i e_l)) - e_i (T
			c^{i,j}_{kj}) = \\ \frac{1}{2} c^{ij,l}_k e_i e_l e_j +
			c^{j}_{ki,l} (e^i (e^l e_j)) -
			\frac{1}{2} (c^{ij}_{k,l} -2 
			c_{kl}^{j,i} ) (e^l (e_i e_j)) 
			 +
			 (Te_i) (c^{ij}_{k,j} -
			 c^{i,j}_{kj}
			 ) + \\ (Te^i) c^j_{ki,j} + \left( T
			c^j_{kl, j} - 2 c^i_{jl} (Tc^j_{ki}) 
			\right) e^l + \left(  2(Tc^j_{ki}) c^{il}_j
			-(Tc^{ij}_k) c^l_{ij} + T c^{lj}_{k,j} -
			Tc^{l,j}_{kj} \right) e_l,
			\label{eq:cuad2}
		\end{multline}
		where for a function $f$, we use the notation
		\[ f_{,i} := \pi(e_i) f, \qquad f^{,i} := \pi(e^i) f.\]
		The $\chi$ term in \ref{eq:cuad1} is simply:
		\begin{equation} 2 Se_k - e^i (c^j_{ki} e_j) - e_i (c^{ij}_ke_j - c^i_{kj}
		e^j). \label{eq:cuad3}
		\end{equation}
	To compute the $\lambda$ term we need to evaluate the
	integral terms in (\ref{eq:cuad1}). For this we compute:
	\begin{equation}
		\int_0^\Lambda [{c_{k}^{ij}
			e_j - c_{kj}^i e^j}_\Gamma Se_i]  d\Gamma  =
			\int_0^\Lambda (S + \eta) [ [e_k, e^i]_\Gamma
			e_i] d\Gamma =  \lambda [ [e_k, e^i], e_i] + 
			\lambda S c_{ki}^i,
		\label{eq:cuad4}
	\end{equation}
	and, 
	\begin{equation}
		\int_0^\Lambda {[[c_{ki}^j e_j}_\Gamma Se^i] d\Gamma =
		\int_0^\Lambda (S + \eta) [ [e_k, e_i]_\Gamma e^i] d
		\Gamma = \lambda [ [e_k, e_i], e^i] - \lambda S c_{ki}^i,
		\label{eq:cuad5}
	\end{equation}
	from where the $\lambda$ term of (\ref{eq:cuad1}) is just
	\begin{equation}
		2 \lambda e_k + [ [e_k, e^i], e_i] + [ [e_k, e_i], e^i].
		\label{eq:cuad6}
	\end{equation}
	We also need to compute
	\begin{multline}
		{[e_k}_\Lambda e^i (e^j [e_i, e_j] )] = [e_k, e^i] (e^j
		[e_i, e_j]) + 2 \chi e^j [e_k, e_j] + e^i \left( [e_k,
		e^j] [e_i,e_j] + 2\chi [e_i, e_k] +\right. \\\left. e^j [e_k, [e_i, e_j]
		] + 2 \lambda \langle [e_k, e^j], [e_i, e_j] \rangle
		\right) + 2 \lambda c_k^{ij} [e_i, e_j] - 2 \lambda e^j
		\langle [e_k, e^i], [e_i, e_j] \rangle = \\ [e_k, e^i]
		(e^j [e_i, e_j]) + e^i ( [e_k, e^j] [e_i, e_j] ) + e^i
		(e^j [e_k, [e_i, e_j]]) + 4 \chi e^i[e_k, e_i] + \\ \lambda
		\left( 4 \langle [e_k, e^j], [e_i, e_j] \rangle e^i + 2
		c^{ij}_k c_{ij}^l e_l \right).
		\label{qub1}
	\end{multline}
	Similarly:
\begin{multline}
	{[e_k}_\Lambda e_i (e_j [e^i, e^j])] = [e_k, e_i] (e_j [e^i,
	e^j]) + e_i \left( [e_k, e_j] [e^i, e^j] + e_j [e_k, [e^i, e^j]]
	+ \right. \\ \left. e_j (2 \chi + S) c_k^{ij} + 2 \lambda \langle
	[e_k, e_j], [e^i, e^j] \rangle \right) - 2 \lambda e_j \langle
	[e_k, e_i], [e^i, e^j] \rangle = \\ [e_k, e_i] (e_j [e^i, e^j]) +
	e_i ([e_k, e_j] [e^i, e^j])) + e_i (e_j [e_k, [e^i, e^j]]) + e_i
	(e_j S
	c^{ij}_k) + \\ 2 \chi e_i (e_j c^{ij}_k) + 4 \lambda \langle [e_k,
	e_j], [e^i, e^j]\rangle e_i.
	\label{qub2}
\end{multline}
 Finally we need 
\begin{multline}
	{[e_k}_\Lambda - i T \mathcal{J} [e^i, e_i]] =  (\lambda + T)
	{[e_k}_\Lambda c^{i}_{ji} e^j + c^{ji}_i e_j] = \\ (\lambda + T)
	\left( c^{i}_{ji,k} e^j + c^{i}_{ji} c_k^{jl} e_l - 
	 c^i_{ji} c_{kl}^j
	e^l 
	+ c^{ji}_{i,k} e_j + c^{ji}_i c_{kj}^l e_l  + 2\chi
	c^i_{ki}\right).
	\label{div1}
\end{multline}
It follows from (\ref{eq:cuad3}) (\ref{qub1}), (\ref{qub2}) and
(\ref{div1}) that the $\chi$ term of ${[e_k}_\Lambda H]$ is given by
\begin{multline}
	S e_k - \frac{1}{2} e^i(c^j_{ki} e_j) - \frac{1}{2} e_i (c^{ij}_k
	e_j ) + \frac{1}{2} e_i (c^{i}_{kj} e^j) + e^i (c_{ki}^j e_j) +
	\\ \frac{1}{2} e_i (e_j c^{ij}_k) + T c^{i}_{ki} = Se_k +
	\frac{1}{2} e^i (c_{ki}^j e_j) + \frac{1}{2} e_i (c^i_{kj} e^j) +
	Tc^i_{ki} = Se_k + Tc^{i}_{ki}.
	\label{eq:tot1}
\end{multline}
We have for the $\lambda$ term:
\begin{multline}
	e_k + \frac{1}{2} [ [e_k, e^i], e_i] + \frac{1}{2} [ [e_k, e_i],
	e^i] +  \langle [e_k, e^j], [e_i, e_j] \rangle e^i + 
	\frac{1}{2} c^{ij}_k c^l_{ij} e_l + \\ \langle [e_k, e_j], [e^i,
	e^j] \rangle e_i + \frac{1}{2} \left(c^{i}_{ji,k} e^j + c^{i}_{ji} c_k^{jl} e_l - 
	 c^i_{ji} c_{kl}^j
	e^l 
	+ c^{ji}_{i,k} e_j + c^{ji}_i c_{kj}^l e_l \right) = \\ e_k +
	\frac{1}{2} [c_k^{ij} e_j, e_i] - \frac{1}{2} [c_{kj}^i e^j, e_i]
	+ \frac{1}{2} [c_{ki}^j e_j, e^i] - c_{kl}^j c_{ij}^l e^i +
	\frac{1}{2} c^{ij}_k c^l_{ij} e_l + c_{kj}^l c^{ij}_l e_i + \\
	\frac{1}{2} \left(c^{i}_{ji,k} e^j + c^{i}_{ji} c_k^{jl} e_l - 
	 c^i_{ji} c_{kl}^j
	e^l 
	+ c^{ji}_{i,k} e_j + c^{ji}_i c_{kj}^l e_l \right)= e_k -
	\frac{1}{2} c^{ij}_{k,i} e_j - \frac{1}{2} c^{ij}_k c_{ij}^l e_l
	+ \\ \frac{1}{2} c^{i}_{kj,i} e^j + \frac{1}{2} c^{i}_{kj}
	c_i^{jl} e_l - \frac{1}{2} c^{i}_{kj} c_{il}^j e^l -
	\frac{1}{4} c^{i}_{ki, j} e^j - \frac{1}{4} c^{i,j}_{ki} e_j -
	\frac{1}{2} c^{j, i}_{ki} e_j - \\ \frac{1}{2} c^j_{ki} c^i_{jl} e^l
	+ \frac{1}{2} c^j_{ki} c^{il}_{j} e_l + \frac{1}{4} c^{i}_{ki, j}
	e^j + \frac{1}{4} c^{i,j}_{ki} e_j - c_{kl}^j c_{ij}^l e^i +
	\frac{1}{2} c^{ij}_k c^l_{ij} e_l + c_{kj}^l c^{ij}_l e_i + \\
	\frac{1}{2} \left(c^{i}_{ji,k} e^j + c^{i}_{ji} c_k^{jl} e_l - 
	 c^i_{ji} c_{kl}^j
	e^l 
	+ c^{ji}_{i,k} e_j + c^{ji}_i c_{kj}^l e_l \right)= e_k - \\ 
	\frac{1}{2} (c^{ij}_{k,i} + c^{j,i}_{ki} -
	c^{ji}_{i,k} - c^{i}_{li} c^{lj}_k - c^{li}_{i} c^{j}_{kl} ) e_j
	+ \frac{1}{2} (c^i_{kj,i}  + c^i_{ji,k} -
	c^i_{li}c^l_{kj} ) e^j.
	\label{tot2}
\end{multline}
To simplify this expression further, recall from (\ref{eq:divergence})
and its dual that we have $\dive_\mu e_k = -c_{ki}^i$ and
$\dive_\mu e^j = -c^{ji}_i$. A simple computation shows
\begin{multline}
	\dive_\mu [e_k, e_j] = \dive_\mu (c_{kj}^i e_i) = c_{kj,i}^i -
	c_{kj}^i c_{il}^l = \dive_\mu \pi [e_k, e_j] = \\ \dive_\mu [\pi e_k,
	\pi e_j] = \pi(e_k) \dive_\mu e_j - \pi(e_j) \dive_\mu e_k =
	- c_{ji,k}^i  + c_{ki,j}^i.
	\label{eq:div2}
\end{multline}
A similar computation shows
\begin{equation}
	\dive_\mu [e_k, e^j] = c_l^{il} c_{ki}^j - c^{ji}_k c^{l}_{il} -
	c_{ki}^{j,i} + c_{k,i}^{ji} = - c^{ji}_{i,k} + c_{ki}^{i,j}.
	\label{eq:div3}
\end{equation}
Replacing (\ref{eq:div3}) and (\ref{eq:div2}) in (\ref{tot2}) we obtain
for the $\lambda$ term of ${[e_k}_\Lambda, H]$:
\begin{equation}
	e_k +\frac{1}{2} c^{i,j}_{ki} e_j +  \frac{1}{2} c_{ki,j}^i e^j =
	e_k + S c_{ki}^i.
	\label{eq:tot3}
\end{equation}
From (\ref{div1}) we find that the $\lambda \chi$ term in
${[e_k}_\Lambda H]$ is simply $c_{ki}^i$ and we need to compute only
the constant term. For this we expand the terms in (\ref{qub1}) which are
cubic in the fermions using quasi-asociativity, a straightforward
computation shows:
\begin{equation}
	\begin{aligned}
		{[e_k}, e^i] (e^j [e_i, e_j]) &= c^{im}_{k} c^{l}_{ij} (e_m
		((e^j e_l)) - c^{i}_{km} c^{l}_{ij} (e^m (e^j e_l)) + \\ &
		\quad +  2
		(Tc^{ij}_k) c^l_{ij} e_l + 2 (Tc^{i}_{kl}) c^l_{ij} e^j, \\
		e^i ([e_k,e^j][e_i, e_j]) &= c^{jl}_k c^m_{ij} (e^i (e_l
		e_m)) - c^j_{kl} c^m_{ij} (e^i (e^l e_m)) - 2 (Tc^j_{kl})
		c^l_{ij} e^i, \\
		e^i (e^j [e_k, [e_i, e_j]]) &= c^l_{ij,k} (e^i (e^j e_l)) +
		c^l_{ij} c^m_{kl} (e^i (e^j e_m)), \\ 
		[e_k, e_i] (e_j[e^i, e^j]) &= c^l_{ki} c^{ij}_m (e_l(e_j
		e^m)) - 2 (Tc^{l}_{ki}) c^{ij}_l e_j, \\
		e_i ([e_k, e_j][e^i, e^j]) &= c^{ij}_m c^{l}_{kj} (e_i
		(e_l e^m)) + 2 (T c^m_{kj}) c_m^{ij} e_i, \\
		e_i (e_j [e_k,[e^i, e^j]]) &= c^{ij}_{l,k} (e_i (e_j e^l))
		- c^{ij}_l c^l_{km} (e_i (e_j e^m)) + c^{ij}_l c^{lm}_k
		(e_i (e_j e_m))\\ 
		&\quad - \frac{1}{2} c^{ij}_{k,l} (e_i (e_j e^l)) -
		\frac{1}{2} c^{ij,l}_k (e_i (e_j e_l)), \\
		e_i (e_j Sc_k^{ij}) &= \frac{1}{2} c^{ij}_{k,l} (e_i (e_j
		e^l)) + \frac{1}{2} c^{ij,l}_{k} (e_i (e_j e_l)).
	\label{eq:tot4}
\end{aligned}
\end{equation}
Collecting all
the terms of the constant term of ${[e_k}_\Lambda H]$ that do not
explicitely contain cubic products of sections of 
$T\oplus T^*$ we get:
\begin{multline}
	\frac{1}{2} (Te_i) \left( c^{ij}_{k,j} -  c^{i,j}_{kj} 
	 \right) + \frac{1}{2} (Te^i) c^j_{ki,j} +
	\frac{1}{2} \left( T
			c^j_{kl, j} - 2 c^i_{jl} (Tc^j_{ki}) 
			\right) e^l + \\ \frac{1}{2}  \left(  2(Tc^j_{ki})
			c^{il}_j
			-(Tc^{ij}_k) c^l_{ij} + T c^{lj}_{k,j} -
			Tc^{l,j}_{kj} \right) e_l + \frac{1}{2} T \left( c^{i}_{ji,k} e^j + c^{i}_{ji} c_k^{jl} e_l - \right. \\
	\left. c^i_{ji} c_{kl}^j
	e^l 
	+ c^{ji}_{i,k} e_j + c^{ji}_i c_{kj}^l e_l \right) +
	\frac{1}{2} (Tc^{ij}_k) c^l_{ij} e_l + \frac{1}{2} (Tc^{i}_{kl})
	c^l_{ij} e^j - \frac{1}{2} (Tc^j_{kl})
	c^l_{ij} e^i - \\ \frac{1}{2} (Tc^l_{ki}) c^{ij}_l e_j +
	\frac{1}{2} (Tc^m_{kj}) c_m^{ij} e_i =   
			  \frac{1}{2} T \left(
			\Bigl( c^{i}_{ji,k}  + c^i_{kj, i}  
			 - c^i_{li} c_{kj}^l \Bigr) e^j \right) +\\
			 \frac{1}{2} T \left( \Bigl(
          c^{i}_{li} c_k^{lj}   
	+ c^{ji}_{i,k}  + c^{li}_i c_{kl}^j + c^{ji}_{k,i} -
			c^{j,i}_{ki}\Bigr) e_j \right) =  
			  \frac{1}{2} T \left(
			  c^{i}_{ki,j} e^j \right) +
			  \frac{1}{2} T \left( c^{i,j}_{ki} e_j \right) =
			   TS c_{ki}^i,
\label{tot5}
\end{multline}
where we have used (\ref{eq:div2}) and (\ref{eq:div3}) in the last line.
 Collecting all the terms that contain cubic products of sections of $L$
 we get:
 \begin{equation}
	 \frac{1}{4} c^{ij,l}_k (e_i (e_l e_j)) 
	  + \frac{1}{4} c^{ij}_l c^{lm}_k (e_i (e_j e_m))
	 =
	 \frac{1}{4} \left( c^{ij}_m c^{ml}_k 
	 - c^{ij,l}_k \right) (e_i (e_j e_l))
	 = 0,
 \end{equation}
 which vanishes because of the Jacobi identity.
Collecting all the terms that are quadratic in sections of $L$ and linear
in $L^*$ we get:
\begin{multline}
	\Bigl(-\frac{1}{4} c^{ij}_{k,l}
	+\frac{1}{2} c^{j,i}_{kl} \Bigr) (e^l (e_i e_j)) + \frac{1}{4} c^{im}_k
	c^l_{ij} (e_m (e^j e_l)) +  \frac{1}{4} c_k^{jl}c_{ij}^m (e^i(e_l
	e_m)) + \\ \frac{1}{4} c^{ij}_m c^l_{kj} (e_i (e_l e^m)) +
	\frac{1}{4} c^{ij}_{l,k} (e_i (e_je^l)) -\frac{1}{4} c^{ij}_l c^l_{km} (e_i
	(e_j e^m)) + 
	\frac{1}{4} c^{l}_{ki} c^{ij}_m (e_l (e_j e^m))= \\
	\Biggl( -\frac{1}{4} c^{ij}_{k,l}
	 +  \frac{1}{2} c_k^{mi}c_{lm}^j + \frac{1}{4} c^{im}_l
	c^j_{km}  +
	\frac{1}{4} c^{ij}_{l,k}  -\frac{1}{4} c^{ij}_m c^m_{kl} 
	+ \frac{1}{2} c^{j,i}_{kl} +
	\frac{1}{4} c^{i}_{km} c^{mj}_l \Biggr) (e_i (e_j e^l))
	=0.
\end{multline}
Finally, collecting all the terms that are quadratic in sections of
$L^*$ and linear in $L$:
\begin{multline}
	\frac{1}{2} c^j_{ki,l} (e^i (e^l e_j)) - \frac{1}{4} c^i_{km}
	c^{l}_{ij} (e^m (e^je_l)) - \frac{1}{4} c_{kl}^j c^m_{ij} (e^i
	(e^l e_m)) + \\ \frac{1}{4} c^l_{ij,k} (e^i (e^j e_l)) +
	\frac{1}{4} c^l_{ij} c^m_{kl} (e^i (e^j e_m)) = \\
	\frac{1}{4} \Bigl( 2 c^l_{ki,j}  -  c^m_{ki}
	c^{l}_{mj} -  c_{kj}^m c^l_{im} 
 	+  c^l_{ij,k}  +
	 c^m_{ij} c^l_{km} \Bigr) (e^i (e^j e_l)),
	\label{jac3}
\end{multline}
which vanishes because of the Jacobi identity.

It follows from (\ref{eq:tot1}), (\ref{eq:tot3}) and
(\ref{tot5})-(\ref{jac3}) that we have
\begin{equation}
	{[e_k}_\Lambda H] = TS c_{ki}^i + \chi (Se_k + Tc_{ki}^i) +
	\lambda (e_k + Sc_{ki}^i) + \lambda \chi c_{ki}^i.
	\label{eq:tot7}
\end{equation}
Using skew-symmetry we obtain:
\begin{equation}
	{[H}_\Lambda e_k] = (2T + \lambda + \chi S) e_k - \lambda \chi
	c_{ki}^i.
	\label{eq:tot8}
\end{equation}
Similarly we find
\begin{equation}
	{[H}_\Lambda e^k] = (2T + \lambda + \chi S) e^k - \lambda \chi
	c^{ki}_i.
	\label{eq:tot9}
\end{equation}
Which proves 2) for the sections $e_k$ and $e^k$. To find 2) for a
general section we just use the non-commutative Wick formula and 1). 

3) We use the non-commutative Wick formula to obtain:
\begin{multline}
	[H_\Lambda e^k e_k] =  \Bigl((2T + \lambda + \chi S)
	e^k\Bigr) e_k -
	\lambda\chi c^{ki}_i e_k + e^k (2 T + \lambda + \chi S) e_k + \\
	\lambda \chi e^k c_{ki}^i  + \int_0^\Lambda (-2\gamma + \lambda
	-\chi \eta) ([e^k, e_k] + 2 \eta \dim M) d\Gamma = \\ (2T +
	2\lambda + \chi S) e^k e_k .
	\label{tot10}
\end{multline}
The Theorem follows easily from (\ref{tot10}) and the 
following simple Lemma that is
interesting on its own: 
\begin{lem}
	Let $V$ be an $N_K=1$ SUSY vertex algebra. Let $J$ and $H$ be two fields of
	$V$ satisfying\footnote{Note that we
		cannot say that $J$ is a primary field of conformal weight $1$ from this
		equation since we do not know yet that $H$ is a superconformal field.}:
	\begin{equation*}
		[J_\Lambda J] = - \left( H + \frac{c}{3} \lambda \chi \right),
		\qquad [H_\Lambda J] = (2 T + 2\lambda + \chi S) J,
	\end{equation*}
	for some complex number $c$. Then the fields $J$ and $H$ generate an $N=2$
	super vertex algebra of central charge $c$, namely:
	\begin{equation*}
		[H_\Lambda H] = (2 T + 3\lambda + \chi S) H + \frac{c}{3} \lambda^2
		\chi.
	\end{equation*}
	\label{lem:technical}
\end{lem}
\begin{proof}
	This is a direct application of the Jacobi identity for $N_K=1$ SUSY Lie
	conformal algebras \cite{heluani3}:
	\begin{multline*}
		 [H_\Lambda H] = - [H_\Lambda [J_\Gamma J]] =- [ [H_\Lambda
		J]_{\Lambda + \Gamma} J] - [J_\Gamma [H_\Lambda J]] = \\- [(2T +
		2\lambda + \chi S) J_{\Lambda + \Gamma} J] - [J_\Gamma (2T +
		2 \lambda + \chi S) J] = \\ + (2\gamma + \chi (\eta + \chi))
		[J_{\Lambda + \Gamma} J] - (2 T + 2(\lambda + \gamma) + \chi (S +
		\eta)) [J_{\Gamma} J] = \\ - (2 \gamma + \chi (\eta + \chi)) \left(
		H + \frac{c}{3} (\lambda + \gamma)(\chi + \eta) \right) + (2T + 2
		(\lambda + \gamma) + \chi (S + \eta)) \left(H + \frac{c}{3} \gamma
		\eta
		\right) = \\ (2 T + 3 \lambda + \chi S) H + \frac{c}{3} \lambda^2
		\chi.
	\end{multline*}
\end{proof}
\renewcommand{\qedsymbol}{}
\end{proof}

\end{document}